\documentclass{amsart}

\usepackage{amssymb,latexsym,amsfonts,amsmath}
\usepackage{graphicx,epsfig}
\include{diagrams}
\usepackage{graphicx}
\usepackage{color}
\usepackage{amsfonts}
\usepackage{amssymb,latexsym,amsfonts,amsmath}
\usepackage{graphicx}
\usepackage{tikz}
\usepackage{subfigure}
\usetikzlibrary{automata}

\usepackage{savesym}
\savesymbol{AND}
\usepackage{algorithmic}
\usepackage{algorithm}

\newtheorem{theorem}{Theorem}[section]
\newtheorem{lemma}[theorem]{Lemma}

\newtheorem{proposition}[theorem]{Proposition}
\newtheorem{corollary}[theorem]{Corollary}

\theoremstyle{definition}
\newtheorem{definition}[theorem]{Definition}
\newtheorem{example}[theorem]{Example}

\theoremstyle{remark}
\newtheorem{remark}[theorem]{Remark}

\numberwithin{equation}{section}
\newtheorem{alg}[]{\bf {Algorithm}}

\begin{document}

\begin{abstract}
In this paper, a general framework is proposed for the analysis and
characterization of observability and diagnosability of finite state systems.
Observability corresponds to the reconstruction of the system's discrete
state, while diagnosability corresponds to the possibility of determining the
past occurrence of some particular states, for example faulty states. A
unifying framework is proposed where observability and diagnosability
properties are defined with respect to a critical set, i.e. a set of discrete
states representing a set of faults, or more generally a set of interest.
These properties are characterized and the involved conditions provide an
estimation of the delay required for the detection of a critical state, of the
precision of the delay estimation and of the duration of a possible initial
transient where the diagnosis is not possible or not required. Our framework
makes it possible to precisely compare some of the observabllity and
diagnosability notions existing in the literature with the ones introduced in
our paper, and this comparison is presented.
\end{abstract}

\title[Observability and Diagnosability of Finite State Systems: a Unifying Framework]{Observability and Diagnosability of Finite State Systems: \\a Unifying Framework}

\thanks{The research leading to these results has been partially supported by the Center of Excellence DEWS}

\author[Elena De Santis and Maria D. Di Benedetto]{
Elena De Santis and Maria D. Di Benedetto
}
\address{$^1$Department of Information Engineering, Computer Science and Mathematics, Center of Excellence DEWS,
University of L{'}Aquila, 67100 L{'}Aquila, Italy}
\email{
\{elena.desantis,mariadomenica.dibenedetto\}@univaq.it}

\maketitle

\section{Introduction}

Reconstructing the internal behavior of a dynamical system on the basis of the
available measurements is a central problem in control theory. Starting from
the seminal paper \cite{kalman:1959}, state observability has been
investigated both in the continuous domain (see e.g. the fundamental papers
\cite{Luenb:71} for the linear case and \cite{griffith:1971} for the nonlinear
case), in the discrete state domain (see e.g. \cite{Ozveren90} and
\cite{Ramadge86}), and more recently for hybrid systems (see e.g. the special
issue \cite{DeSDiB:IJRNC(ed)} on observability and observer-based control of
hybrid systems and the references therein, \cite{Balluchi02},
\cite{Bemporad:tac2000}, \cite{Collins:2007}, \cite{baab:2005},
\cite{DeS:CDC03}, \cite{Vidal:2003}, \cite{Bal:Aut13}, \cite{Tanw:AC13}). In
some references dealing with discrete event systems, e.g. in
\cite{Lafortune:2007}, the notion of observability is related to state
disambiguation, which is the property of distinguishing unambiguously among
certain pairs of states in the state space. We will use here the term
observability in the traditional meaning used in \cite{Zaytoon:2013} where
observability corresponds to the reconstruction of the system's discrete
state. Diagnosability, a property that is closely related to observability but
is more general, corresponds to the possibility of detecting the occurrence of
some particular state, for example a faulty state, on the basis of the
observations. An excellent survey of recent advances on diagnosis methods for
discrete systems can be found in \cite{Zaytoon:2013}. The formal definition
and analysis of observability and diagnosability depend on the model, on the
available output information, and on the objective for which state
reconstruction is needed, e.g. for control purposes, for detection of critical
situations, and for diagnosis of past system evolutions. It is therefore hard,
in general, to understand the precise relationships that exist between the
different notions that exist in the literature.

In this paper, we propose a unifying framework where observability and
diagnosability are defined with respect to a subset of the state space, called
critical set. A state belonging to the critical set is called critical state.
This idea comes from safety critical applications, e.g. Air Traffic Management
\cite{DiBenedetto:CDC2005}, \cite{Des:Springer06}, where the critical set of
discrete states represents dangerous situations that must be detected to avoid
unsafe or even catastrophic behavior of the system. However, the critical set
can represent a set of faults, or more generally any set of interest. We
define and characterize observability and diagnosability in a uniform
set-membership-based formalism. The set-membership formalism and the derived
algorithms are very simple and intuitive, and allow checking the properties
without constructing an observer, thereby avoiding the exponential complexity
of the observer design. The definitions of observability and diagnosability
are given in a general form that is parametric with respect to the delay
required for the detection of a critical state, and the precision of the delay
estimation. Using the proposed conditions that characterize those properties,
we can check diagnosability of a critical event, such as a faulty event, and
at the same time compute the delay of the diagnosis with respect to the
occurrence of the event, the uncertainty about the time at which that event
occurred, and the duration of a possible initial transient where the diagnosis
is not possible or not required. These evaluations are useful to better
understand the characteristics of the system and can be used in the
implementation of the diagnoser.

While in the literature on discrete event systems a transition-based model is
used, we adopt a state-based approach, similarly to what was done in
\cite{Lin:1994} where an on-line diagnosability problem for a deterministic
Moore automaton with partial state observation was solved, in \cite{Zad:2003}
where the focus was on the complexity reduction in the diagnoser design, and
in \cite{Takai:AC12} where verification of codiagnosability is performed.
Because of the different formalism used in the transition-based and
state-based approaches, a comparison between our definitions and those
existing in the literature on discrete event systems is very hard to achieve
without a unifying framework where the different notions can all be formulated
and compared. We show that, using our formalism, we are able to understand the
precise relationships that exist between the properties we analyze and some of
the many diagnosability concepts that exist in the literature.

The paper is organized as follows. After introducing the main definitions in
Section \ref{sec1}, Section \ref{sec2} is devoted to establishing some
geometrical tools that are instrumental in proving our results. In Section
\ref{sec3}, observability and diagnosability properties are completely
characterized. The proofs of the main theorems are constructive and show how a
diagnoser can be determined. Some examples are described in Section
\ref{examples}. Finally, in the Appendix we present an extension of some
results of the paper under milder technical assumptions.

\textbf{Notations: }\ The symbol $\mathbb{Z}$ denotes the set of nonnegative
integer numbers. For $a,b\in\mathbb{Z}$, $\left[  a,b\right]$ denotes the set  $\left[  a,b\right]=\left\{
x\in\mathbb{Z}:a\leq x\leq b\right\}  $. For a set $X$, the symbol $\left\vert
X\right\vert $ denotes its cardinality. For a set $Y\subset X$, where the
symbol $\subset$ has to be understood as "subset", not necessarily strict, the
symbol $\overline{Y}$ denotes the complement of $Y$ in $X$, i.e. $\overline
{Y}=\left\{  x\in X:x\notin Y\right\}  $. For $W\subset X\times X$, the symbol
$W^{-}$ denotes the symmetric closure of $W$, i.e. $W^{-}=\left\{  \left(
x_{1},x_{2}\right)  :\left(  x_{1},x_{2}\right)  \in W\text{ or }\left(
x_{2},x_{1}\right)  \in W\right\}  $. The null event is denoted by $\epsilon$.
For a string $\sigma$, $\left\vert \sigma\right\vert $ denotes its length,
$\sigma(i)$, $i\in\left\{  1,2,...,\left\vert \sigma\right\vert \right\}  $,
denotes the $i-th$ element, and $\left\vert \sigma\right\vert _{\left[
a,b\right]  }$ is the string $\sigma(a)\sigma(a+1)...\sigma(b)$. $P\left(
\sigma\right)  $ is the projection of the string $\sigma$, i.e. the string
obtained from $\sigma$ by erasing the symbol $\epsilon$ (see e.g.
\cite{Ramadge:89}). In all figures, the cardinal number inside the circle
denotes the state, the lowercase letter besides the same circle denotes the
output associated to that state.

\section{\label{sec1}Diagnosability properties and their
relationships\label{Definitions}}

We consider a Finite State Machine (FSM)
\[
M=\left(  X,X_{0},Y,H,\Delta\right)
\]
where:

- $X$ is the finite set of states;

- $X_{0}\subset X$ is the set of initial states;

- $Y$ is the finite set of outputs;

- $H:X\rightarrow Y$ is the output function;

- $\Delta\subset X\times X$ is the transition relation.

For $i\in X$, define $succ\left(  i\right)  =\left\{  j\in X:\left(
i,j\right)  \in\Delta\right\}  $ and $pre\left(  i\right)  =\left\{  j\in
X:\left(  j,i\right)  \in\Delta\right\}  $.

We make the following standard assumption:

\begin{description}
\item[Assumption 1] (liveness) $succ(i)\neq\emptyset$, $\forall i\in X$.
\end{description}

Any finite or infinite string $x$ with symbols in $X$ that satisfies the
condition%
\begin{equation}%
\begin{array}
[c]{c}%
x\left(  1\right)  \in X\\
x\left(  k+1\right)  \in succ\left(  x\left(  k\right)  \right)
,\hspace{0.5cm}k=1,2,...,\left\vert x\right\vert -1
\end{array}
\label{state}%
\end{equation}
is called a state execution (or state trajectory or state evolution) of the
FSM $M$. The singleton $\left\{  i\in X\right\}  $ is an execution.

Let $\mathcal{X}^{\ast}$ be the set of all the state executions of $M$. Then,
for a given $\Psi\subset X$, we can define the following subsets of
$\mathcal{X}^{\ast}$:

\qquad - $\mathcal{X}_{\Psi}$ is the set of state executions $x\in
\mathcal{X}^{\ast}$ with $x\left(  1\right)  \in\Psi$

\qquad - $\mathcal{X}_{\Psi,\infty}$ is the set of infinite state executions
$x\in\mathcal{X}^{\ast}$ with $x\left(  1\right)  \in\Psi$. For simplicity,
the set $\mathcal{X}_{X_{0},\infty}$ will be denoted by $\mathcal{X}$

\qquad - $\mathcal{X}^{\Psi}$ is the set of finite state executions
$x\in\mathcal{X}^{\ast}$ with last symbol in $\Psi$

Obviously,
\[
\mathcal{X}_{X}=\mathcal{X}^{\ast}%
\]
and
\[
\mathcal{X}_{\Psi,\infty}\mathcal{\subset X}_{\Psi}\subset\mathcal{X}^{\ast}%
\]

Let $\mathcal{Y}$ be the set of strings with symbols in $\widehat{Y}=\left\{
y\in Y:y\neq\epsilon\right\}  $. Define $\mathbf{y}:\mathcal{X}^{\ast
}\rightarrow\mathcal{Y}$, the function that associates to a state execution
the corresponding output execution, as
\[
\mathbf{y}\left(  x\right)  =P\left(  \sigma\right)
\]
where
\[
\sigma=H\left(  x(1)\right)  ...H\left(  x(n)\right)  ,n=\left\vert
x\right\vert
\]
if $\left\vert x\right\vert $ is finite. Otherwise
\[
\mathbf{y}\left(  x\right)  =P\left(  \sigma_{\infty}\right)
\]
where $\sigma_{\infty}$ is an infinite string recursively defined as
\begin{align*}
\sigma_{1}  &  =H\left(  x(1)\right) \\
\sigma_{k+1}  &  =\sigma_{k}H\left(  x(k+1)\right)  ,\text{ }k=1,2,...
\end{align*}

Finally, for $x\in\mathcal{X}_{X_{0}}$
\[
\mathbf{y}^{-1}\left(  \mathbf{y}\left(  x\right)  \right)  =\left\{
\widehat{x}\in\mathcal{X}_{X_{0}}:\mathbf{y}\left(  \widehat{x}\right)
=\mathbf{y}\left(  x\right)  \right\}
\]


We now propose a framework where observability and diagnosability are defined
with respect to a subset of the state space $\Omega\subset X$ called critical
set. The set $\Omega$ may represent unsafe states, faulty states, or more
generally any set of states of interest.

For a string $x\in\mathcal{X}$, two cases are possible:

$i)$ $x\left(  k\right)  \notin\Omega$, $\forall k\in\mathbb{Z}$

$ii)$ $x\left(  k\right)  \in\Omega$, for some $k\in\mathbb{Z}$

If the second condition holds, let $k_{x}$ be the minimum value of $k$ such
that $x\left(  k\right)  \in\Omega$, i.e.%
\begin{gather}
k_{x}=k\in\mathbb{Z}:x\left(  k\right)  \in\Omega\nonumber\\
\text{and}\label{kx}\\
(k=1\text{ or }x(h)\notin\Omega,\forall h\in\left[  1,k-1\right]  )\nonumber
\end{gather}
Otherwise set $k_{x}=\infty$.

The next definition describes the capability of inferring, from the output
execution, that the state belongs to the set $\Omega$, at some step during the
execution, after a finite transient or after a finite delay or with some
uncertainty in the determination of the step. The precise meaning of the
parameters used to describe those characteristics will be discussed after the definition.

\begin{definition}
\label{Def_somehow}The FSM $M$ is parametrically diagnosable with respect to a
set $\Omega\subset X$ (shortly parametrically $\Omega-diag$) if there exist
$\tau$ and $\delta\in\mathbb{Z}$, and $T\in\mathbb{Z}\cup\left\{
\infty\right\}  $ such that for any string $x\in\mathcal{X}$ with finite
$k_{x}$, whenever $x\left(  k\right)  \in\Omega$ and $k\in\left[  \max\left\{
k_{x},\left(  \tau+1\right)  \right\}  ,k_{x}+T\right]  $, it follows that for
any string $\widehat{x}\in\mathbf{y}^{-1}\left(  \mathbf{y}\left(  \left.
x\right\vert _{\left[  1,k+\delta\right]  }\right)  \right)  $, $\widehat{x}%
\left(  h\right)  \in\Omega$, for some $h\in\left[  \max\left\{  1,\left(
k-\gamma_{1}\right)  \right\}  ,k+\gamma_{2}\right]  $ and for some
$\gamma_{1},\gamma_{2}\in\mathbb{Z}$, $\gamma_{2}\leq\delta$.
\end{definition}

If $x\left(  k\right)  \in\Omega$ for some $k\in\mathbb{Z}$, in what follows
the condition $x\left(  k\right)  \in\Omega$ is called \emph{crossing event},
and $k$ is the step at which the crossing event occurs.

The value $\gamma=\max\left\{  \gamma_{1},\gamma_{2}\right\}  $ is the
uncertainty radius in the reconstruction of the step at which the crossing
event occurred. The parameter $\delta$ corresponds to the delay of the
crossing event detection while $\tau$ corresponds to an initial time interval
where the crossing event is not required to be detected.

The detection of the crossing event is required whenever it occurs in the
interval defined by the parameter $T$.

To better understand the role of these parameters, consider the examples in
Figure \ref{fig1}. For fixed values $\tau$, $T$, $\delta$ and $\gamma$, we
have represented three possible cases, corresponding to three different
executions, and hence with different values for $k_{x}$. In the first case
$\max\left\{  k_{x},\left(  \tau+1\right)  \right\}  =\left(  \tau+1\right)
$. Hence, any crossing event occurring in $\left[  \left(  \tau+1\right)
,k_{x}+T\right]  $ has to be detected, with maximum delay $\delta$ and with
maximum uncertainty $\gamma$. Crossing events occurring in $\left[
1,\tau\right]  $ are not needed to be detected. In \ the second case,
$\max\left\{  k_{x},\left(  \tau+1\right)  \right\}  =k_{x}$, and therefore
any crossing event up to step $k_{x}+T$ has to be detected, with maximum delay
$\delta$ and with maximum uncertainty $\gamma$. Finally in the last case no
detection is required.%
\begin{figure}[ptb]
\begin{center}
\includegraphics[scale=0.5]{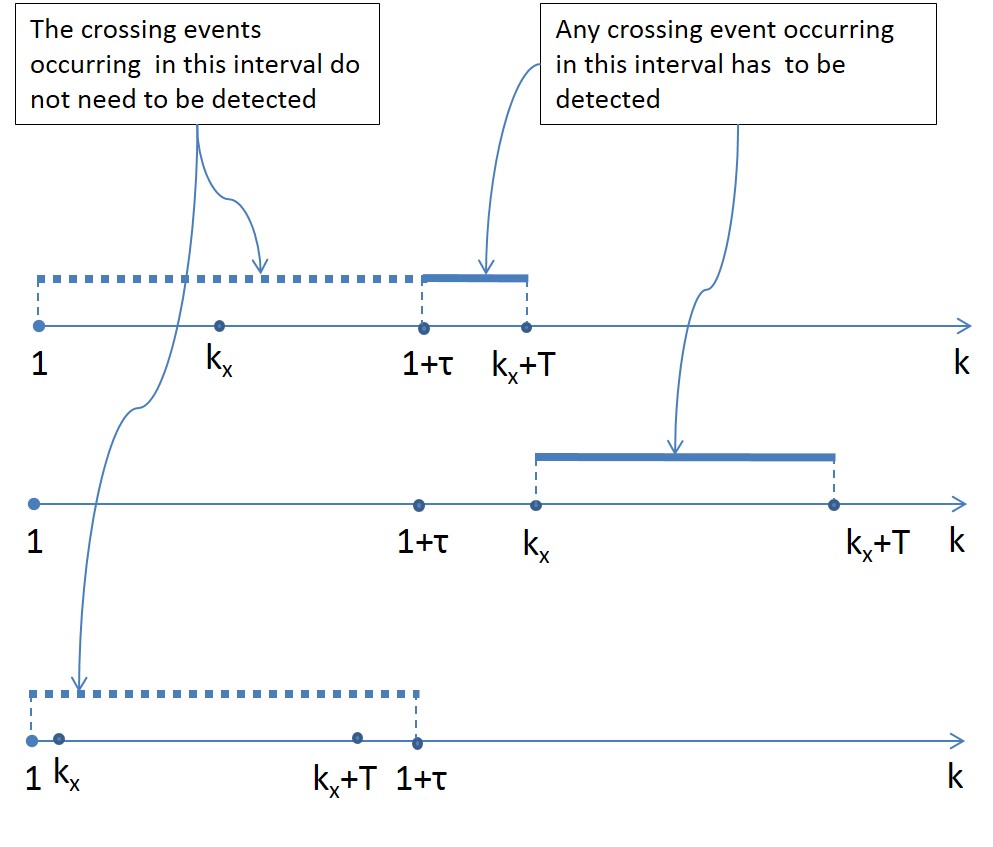}
\end{center}
\caption{Illustration of parameters $k_{x}$, $\tau$ and $T$}%
\label{fig1}%
\end{figure}

Obviously $T=0$ and $\tau=0$ mean that only the first crossing event has to be
detected. Moreover, $\delta=0$ implies $\gamma=0$, but $\gamma=0$ does not
imply in general $\delta=0$.%

\begin{figure}[ptb]
\begin{center}
\includegraphics[scale=0.5]{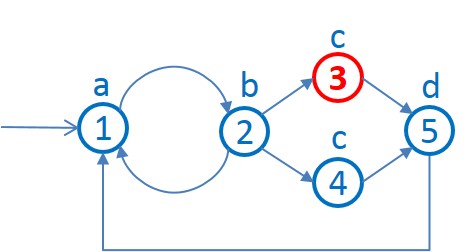}
\end{center}
\caption{The FSM is not parametrically $\left\{  3\right\}  -diag$.}%
\label{fig_nodiag}%
\end{figure}

\begin{example}
As an example, the FSM in Figure \ref{fig_nodiag}, where $\Omega=\left\{
3\right\}  $, is not parametrically $\Omega-diag$: in fact for any $\tau$
there exists a state execution that crosses the set $\Omega$ for the first
time at some $k>\tau$, and it is not possible to detect the crossing event
neither immediately nor with delay, neither exactly nor with uncertainty.
\end{example}

By definition of parametric diagnosability, the following monotonicity
property holds:

\begin{proposition}
\label{prop_mono}If $M$ is parametrically $\Omega-diag$ with parameters $\tau
$, $\delta$, $T$, $\gamma_{1}$, $\gamma_{2}$ then$\ $it is parametrically
$\Omega-diag$ with parameters $\tau^{\prime}$, $\delta^{\prime}$, $T^{\prime}%
$, $\gamma_{1}^{\prime}$, $\gamma_{2}^{\prime}$, where $\tau^{\prime}\geq\tau
$, $\delta^{\prime}\geq\delta$, $T^{\prime}\leq T$, $\gamma_{1}^{\prime}%
\geq\gamma_{1}$, $\gamma_{2}^{\prime}\geq\gamma_{2}$, $\gamma_{2}^{\prime}%
\leq\delta^{\prime}$.
\end{proposition}

Depending on the values taken by $\tau$, $\delta$ and $T$, special instances
of Definition \ref{Def_somehow} are obtained. We consider the following, which
highlight the role of these parameters:
\begin{equation}%
\begin{tabular}
[c]{llll}%
$a.$ & $T=0$ & $\tau=0$ & $\delta\geq0$\\
$b.$ & $T=\infty$ & $\tau>0$ & $\delta>0$\\
$c.$ & $T=\infty$ & $\tau>0$ & $\delta=0$\\
$d.$ & $T=\infty$ & $\tau=0$ & $\delta>0$\\
$e.$ & $T=\infty$ & $\tau=0$ & $\delta=0$%
\end{tabular}
\ \ \ \ \ \label{special_cases}%
\end{equation}

\begin{description}
\item[case a.] Since $T=0$ the crossing event can be detected the first time
it occurs, immediately or with some delay. If $\Omega\subset X_{0}$, and
$\gamma_{1}=\gamma_{2}=0$, case a. becomes an extension of the definition of
initial state observability property, as given e.g. in \cite{Ramadge86}, which
we call here $\Omega-$\emph{initial state observability}.

\item[case b.] In this case, the crossing event can be detected with a maximum
delay of $\delta$ steps whenever it occurs for $k\geq\tau+1$. However, since
$\tau\geq1$, the event $x\left(  k\right)  \in\Omega$, $k\in\left[
1,\tau\right]  $ may not be detected (in this case, parametric diagnosability
is an "eventual" property). The notion of $\left(  k_{1},k_{2}\right)
-$\emph{detectability}, as introduced in \cite{Shu:AC2013} can be retrieved as
a special case. In fact it corresponds to parametric $\left\{  x\right\}
-$diagnosability, $\forall x\in X$, with parameters $T=\infty$, $\tau=k_{1}$,
$\delta=k_{2}$ and $\gamma=0$.

\item[case c.] With respect to case $b.$, in this case the crossing event can
be detected without any delay, and the FSM $M$ is said to be $\Omega
-$\emph{current state observable.} It is again an eventual property. Moreover,
$M$ is said to be \emph{current state observable} if it is $\left\{
x\right\}  -$current state observable, $\forall x\in X$. The notion of current
state observability coincides with the one studied in \cite{Bal:Aut13} and
with the notion of Strong Detectability as defined in \cite{Shu:2007}.
Finally, if $\Omega=\left\{  x\right\}  $ and $M$ is $\Omega-$current state
observable, then the state $x$ is \emph{always observable}, as defined in
\cite{Ozveren90}.

\item[case d.] The meaning is the same as in case $b.$, but with $\tau=0$.
\ In this case, the FSM $M$ is said to be \emph{critically diagnosable} with
respect to $\Omega$ (critical $\Omega-diag$). Critical diagnosability is an
"always" property.

\item[case e.] The meaning is the same as in case $c.$, but with $\tau=0$. The
FSM $M$ is said to be\ \emph{critically observable} with respect to $\Omega$
(critical $\Omega-obs$). It is again an "always" property. The notion of
critical observability for an FSM was introduced in \cite{DiBenedetto05c} and
\cite{DiBenedetto:CDC2005}. In \cite{Des:Springer06} the same notion was
extended to linear switching systems with minimum and maximum dwell time.
Finally, in \cite{PolaAutom2016} the analysis of critical observability was
extended to the case of networks of Finite State Machines.
\end{description}

For an exhaustive analysis, in addition to the cases above, let's consider
also the case when $T$ is finite and nonzero. This case deserves some
attention only if we require the exact reconstruction of the step at which the
crossing event occurs, i.e. $\gamma=0$. In fact, if $M$ is parametrically
$\Omega-diag$ with parameters $\tau$, $\delta$, $T=0$, $\gamma_{1}$,
$\gamma_{2}$, then, by Definition \ref{Def_somehow} and by Proposition
\ref{prop_mono} it is parametrically $\Omega-diag$ also with parameters
$\tau+\widehat{T}$, $\delta+\widehat{T}$, $T=\widehat{T}$, $\gamma
_{1}+\widehat{T}$, $\gamma_{2}+\widehat{T}$, for any finite $\widehat{T}%
\in\mathbb{Z}$. Conversely, by Definition \ref{Def_somehow},\ if $M$ is
parametrically $\Omega-diag$ with parameters $\tau$, $\delta$, $T=\widehat{T}%
$, $\gamma_{1}$, $\gamma_{2}$ then it is $\Omega-diag$ also with parameters
$\tau$, $\delta$, $T=0$, $\gamma_{1}$, $\gamma_{2}$. The characterization of
the property in the case $T$ finite and nonzero and $\gamma=0$ is a
generalization of case $a$, which is not explicitly addressed in this paper.

For simplicity of exposition, we now define three properties that correspond
to case $a.$, cases $b.$ and $c.$, cases $d.$ and $e.$, respectively.

\begin{definition}
\label{def_oneshot}(case a.) The FSM $M$ is \emph{diagnosable} with respect to
a set $\Omega\subset X$ ($\Omega-diag$) if there exists $\delta\in\mathbb{Z}$,
such that for any $x\in\mathcal{X}$ for which $k_{x}\neq\infty$, it follows
that for any string $\widehat{x}\in\mathbf{y}^{-1}\left(  \mathbf{y}\left(
\left.  x\right\vert _{\left[  1,k_{x}+\delta\right]  }\right)  \right)  $,
$\widehat{x}\left(  h\right)  \in\Omega$, for some $h\in\left[  \max\left\{
1,k_{x}-\gamma_{1}\right\}  ,k_{x}+\gamma_{2}\right]  $ and for some
$\gamma_{1},\gamma_{2}\in\mathbb{Z}$, $\gamma_{2}\leq\delta$. If the property
holds with $\delta=0$, $M$ will be called \emph{observable} with respect to a
set $\Omega\subset X$ ($\Omega-obs$). If $\Omega\subset X_{0}$ and the
property holds with $\gamma_{1}=\gamma_{2}=0$, $M$ will be called $\Omega
-$\emph{initial state observable.}
\end{definition}

The $\Omega-diag$ property of Definition \ref{def_oneshot} corresponds to the
one given in \cite{Sampath95}, where only the detection of the condition
$x\left(  h\right)  \in\Omega$, for some $h\in\left[  1,k_{x}+\delta\right]
$, is required but not the refinement of the identification of the step at
which the crossing event occurs. However, we will show in Section
\ref{omega_diag} that these two properties are equivalent, i.e. an FSM enjoys
the property in \cite{Sampath95} if and only if the requirements of Definition
\ref{def_oneshot} hold.

\begin{definition}
\label{Def_eventual}(cases b. and c.) The FSM $M$ is \emph{eventually
diagnosable} with respect to a set $\Omega\subset X$ (eventually $\Omega
-diag$) if there exist $\tau$ and $\delta\in\mathbb{Z}$ such that for any
string $x\in\mathcal{X}$ with finite $k_{x}$, whenever $x\left(  k\right)
\in\Omega$ and $k\geq\max\left\{  k_{x},\left(  \tau+1\right)  \right\}  $, it
follows that for any string $\widehat{x}\in\mathbf{y}^{-1}\left(
\mathbf{y}\left(  \left.  x\right\vert _{\left[  1,k+\delta\right]  }\right)
\right)  $, $\widehat{x}\left(  h\right)  \in\Omega$, for some $h\in\left[
\max\left\{  1,k-\gamma_{1}\right\}  ,k+\gamma_{2}\right]  $ and for some
$\gamma_{1},\gamma_{2}\in\mathbb{Z}$, $\gamma_{2}\leq\delta$. If the condition
holds with $\delta=0$, $M$ will be called \emph{eventually observable} with
respect to a set $\Omega\subset X$ (eventually $\Omega-obs$).
\end{definition}

Finally, we state the following definition

\begin{definition}
\label{Def_critical}(cases d. and e.) The FSM $M$ is \emph{critically
diagnosable} with respect to a set $\Omega\subset X$ (critically $\Omega
-diag$) if there exists $\delta\in\mathbb{Z}$, such that for any string
$x\in\mathcal{X}$ with finite $k_{x}$, whenever $x\left(  k\right)  \in\Omega
$, it follows that for any string $\widehat{x}\in\mathbf{y}^{-1}\left(
\mathbf{y}\left(  \left.  x\right\vert _{\left[  1,k+\delta\right]  }\right)
\right)  $, $\widehat{x}\left(  h\right)  \in\Omega$, for some $h\in\left[
\max\left\{  1,k-\gamma_{1}\right\}  ,k+\gamma_{2}\right]  $, and for some
$\gamma_{1},\gamma_{2}\in\mathbb{Z}$, $\gamma_{2}\leq\delta$. \ If the
condition holds with $\delta=0$, the FSM $M$ is called \emph{critically
observable} with respect to $\Omega\subset X$ (critically $\Omega-obs$)
\end{definition}

The following relationship can be established between the diagnosability
properties introduced above.

\begin{proposition}
\label{prop}$M$ is critically $\Omega-diag$ if and only if it is $\Omega-diag$
and eventually $\Omega-diag$.
\end{proposition}

\begin{proof}
The necessity is obvious. Sufficiency: suppose that $M$ is eventually
$\Omega-diag$ with parameters $\gamma_{1}^{\prime}$, $\gamma_{2}^{\prime}$,
$\tau^{\prime}$ and $\delta^{\prime}$ and that it is $\Omega-diag$ with
parameter $\delta^{"}$. Then, $M$ is eventually $\Omega-diag$ with parameters
$\tau=0$, $\delta=\max\left\{  \tau^{\prime},\delta^{\prime},\delta
^{"}\right\}  $, $\gamma_{1}=\max\left\{  \tau^{\prime},\gamma_{1}^{\prime
}\right\}  $, $\gamma_{2}=\delta=\max\left\{  \tau^{\prime},\delta^{\prime
},\delta^{"}\right\}  $, and hence it is critically $\Omega-diag$.
\end{proof}

We will characterize the properties in Definitions \ref{Def_somehow},
\ref{def_oneshot} and \ref{Def_eventual}. The characterization of the property
in Definition \ref{Def_critical} will follow by Proposition \ref{prop} as a
simple corollary.

\begin{remark}
If $X_{0}=X$ and $\tau=\delta=0$, Definitions \ref{def_oneshot} and
\ref{Def_eventual} become trivial since they correspond in that case to an
instantaneous detection of the crossing event, i.e. $H(i)\neq H(j)$, $\forall
i,j$ such that $i\in\Omega$ and $j\notin\Omega$.
\end{remark}

We end this section with two examples. The first is an FSM $M$ which is
eventually $\Omega-diag$, with $\tau=1$, $\delta=1$, but not with $\tau=0$ or
$\delta=0$. The second shows an FSM which is eventually $\Omega-diag$, with
$\tau=1$, $\delta=2$, $\gamma=1$, but not with $\tau=1$, $\delta=2$ and
$\gamma=0$.

\begin{example}
\label{ex1}Let $M=(X,X_{0},Y,H,\Delta)$, $X=X_{0}=\left\{
1,2,3,4,5,6\right\}  $, $Y=\left\{  a,b,c\right\}  $, $H\left(  1\right)
=H\left(  3\right)  =H\left(  5\right)  =a$, $H\left(  2\right)  =H\left(
4\right)  =b$, $H\left(  6\right)  =c$,
\[
\Delta=\left\{  \left(  1,6\right)  ,\left(  2,1\right)  ,\left(  2,3\right)
,\left(  6,2\right)  ,\left(  3,4\right)  ,\left(  4,6\right)  ,\left(
5,4\right)  ,\left(  6,5\right)  \right\}
\]
be represented in Figure \ref{fig4}.%

\begin{figure}[ptb]
\begin{center}
\includegraphics[scale=0.5]{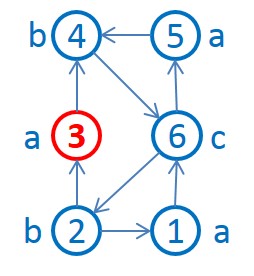}
\end{center}
\caption{FSM $M$ (Example \ref{ex1}).}%
\label{fig4}%
\end{figure}

Let $\Omega=\left\{  3\right\}  $. $M$ is not eventually $\Omega-diag$, with
$\delta=0$. In fact for any state execution ending in state $3$ there is a
state execution ending in state $1$, with the same output string. $M$ is not
eventually $\Omega-diag$, with $\tau=0$. In fact for any state execution
starting from $3$ there is a state execution starting from $5$, with the same
output string. $M$ is eventually $\Omega-diag$, with $\tau=1$, $\delta=1$ and
$\gamma=0$: in fact any output finite string ending with the string $"bab"$
allows the detection of the crossing event and the step at which the crossing occurred.
\end{example}

\begin{example}
\label{ex2}Let $M=(X,X_{0},Y,H,\Delta)$, $X=X_{0}=\left\{
1,2,3,4,5,6,7\right\}  $, $Y=\left\{  a,c\right\}  $, $H\left(  i\right)  =a$,
$i=1...5$, $H\left(  i\right)  =c$, $i=6,7$ and $\Delta$ equal to the set
\[
\left\{  \left(  1,2\right)  ,\left(  2,3\right)  ,\left(  3,6\right)
,\left(  1,4\right)  ,\left(  4,5\right)  ,\left(  5,6\right)  ,\left(
6,6\right)  ,\left(  6,7\right)  ,\left(  7,1\right)  \right\}
\]
be represented in Figure \ref{fig5}.%

\begin{figure}[ptb]
\begin{center}
\includegraphics[scale=0.5]{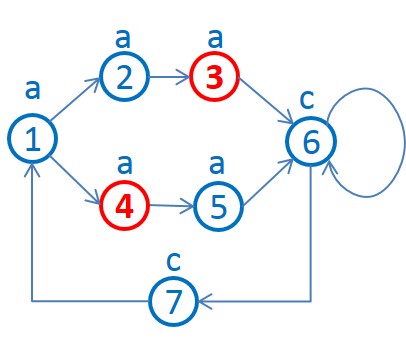}
\end{center}
\caption{FSM $M$ (Example \ref{ex2}).}%
\label{fig5}%
\end{figure}

Let $\Omega=\left\{  3,4\right\}  $. By inspection, we see that $M$ is
eventually $\Omega-diag$, with $\tau=1$, $\delta=2$, $\gamma_{1}=1$ and
$\gamma_{2}=1$. It is not eventually $\Omega-diag$ with $\gamma_{1}=0$ and
$\gamma_{2}=0$. This means that it is possible to detect the crossing event,
but not the step at which the crossing occurred. Note that if $X_{0}=\left\{
1,2,4\right\}  $, then $M$ is eventually $\Omega-diag$, with $\tau=0$,
$\delta=2$, $\gamma_{1}=1$ and $\gamma_{2}=1$, and hence it is critically
$\Omega-diag$.
\end{example}

\section{\label{sec2}Indistinguishability notions}

In what follows, we will assume that the set of outputs does not contain the
null event $\epsilon$.

\begin{description}
\item[Assumption 2:] $\epsilon\notin Y$.
\end{description}

If an FSM $M$ with $\epsilon\in Y$ is such that any cycle has at least a state
$i$ with $H(i)\neq\epsilon$, then we can define an FSM $\widehat{M}$ with
$\epsilon\notin Y$ such that checking the parametric diagnosability property
for $M$ is equivalent to checking the parametric diagnosability property for
$\widehat{M}$. The parameters for which the two properties are satisfied will
in general be different for $M$ and $\widehat{M}$. Details about this
equivalence can be found in the Appendix.

Given the FSM $M=\left(  X,X_{0},Y,H,\Delta\right)  $ and the set $\Omega$,
define the sets
\[
\Pi=\left\{  \left(  i,j\right)  \in X\times X:H(i)=H(j)\right\}
\]
and
\[
\Theta=\left\{  \left(  i,j\right)  \in X\times X:i=j\right\}  \subset\Pi
\]

By definition, the set $\Pi$ and all its subsets are symmetric.

We will refer to the following indistinguishability notions.

\begin{definition}
\label{def_sets} Two state trajectories $x_{1}$ and $x_{2}$ in $\mathcal{X}%
^{\ast}$ are called indistinguishable if $\mathbf{y}\left(  x_{1}\right)
=\mathbf{y}\left(  x_{2}\right)  $. The pair $\left(  i,j\right)  \in\Pi$ is
$k-$forward indistinguishable if there exist $x_{1}\in\mathcal{X}_{\left\{
i\right\}  }$ and $x_{2}\in\mathcal{X}_{\left\{  j\right\}  }$, such that
$\left\vert x_{1}\right\vert =\left\vert x_{2}\right\vert =k$ and
$\mathbf{y}\left(  x_{1}\right)  =\mathbf{y}\left(  x_{2}\right)  $. The pair
$\left(  i,j\right)  \in\Sigma\subset\Pi$ is $k-$backward indistinguishable in
$\Sigma$ if there exist $x_{1}\in\mathcal{X}^{\left\{  i\right\}  }$ and
$x_{2}\in\mathcal{X}^{\left\{  j\right\}  }$, such that $\left\vert
x_{1}\right\vert =\left\vert x_{2}\right\vert =k$, $x_{1}\left(  h\right)
\in\Sigma$, $x_{2}\left(  h\right)  \in\Sigma$, $\forall h\in\left[
1,k\right]  $, and $\mathbf{y}\left(  x_{1}\right)  =\mathbf{y}\left(
x_{2}\right)  $.
\end{definition}

\begin{remark}
Indistinguishability was defined in \cite{Ramadge86} with a different meaning.
By recasting the definitions of \cite{Ramadge86} in our framework, two states
$\ i$ and $j$ were said to be indistinguishable if $\mathcal{X}_{\left\{
i\right\}  }=\mathcal{X}_{\left\{  j\right\}  }$. In the same paper, two
forward indistinguishable states as in Definition \ref{def_sets} were called
\emph{possibly indistinguishable}, while two backward indistinguishable states
were called \emph{possibly indistinguishable} with respect to a
FSM\ associated to the given FSM $M$, called \emph{reverse} FSM
\end{remark}

The following subsets of $\Pi$ will be instrumental in characterizing the
diagnosability properties described in Definitions \ref{Def_somehow},
\ref{def_oneshot}, \ref{Def_eventual} and \ref{Def_critical}.

\begin{itemize}
\item $S^{\ast}\subset\Pi$: set of pairs of states reachable from $X_{0}$ with
two indistinguishable state evolutions

\item $F^{\ast}\subset\Pi$: set of forward indistinguishable pairs of states

\item $B^{\ast}\left(  \Sigma\right)  \subset\Sigma\subset\Pi$: set of
backward indistinguishable pairs of states that belong to a given set $\Sigma$

\item $\Lambda^{\ast}\subset\left(  F^{\ast}\cap S^{\ast}\right)  $: the set
of pairs $\left(  i,j\right)  \in\Pi$, with $i\in\Omega$ and $j\in
\overline{\Omega}$ (or vice-versa $i\in\overline{\Omega}$ and $j\in\Omega$)
for which there exist two indistinguishable infinite state trajectories
starting from $\left\{  i\right\}  $ and $\left\{  j\right\}  $, respectively,
such that the latter is contained in $\overline{\Omega}$ (or vice-versa the
former is contained in $\overline{\Omega}$).

\item $\Gamma^{\ast}\subset B^{\ast}\left(  S^{\ast}\right)  \subset S^{\ast}%
$: set of pairs $\left(  i,j\right)  \in\Pi$, with $i\in\Omega$ and
$j\in\overline{\Omega}$ (or vice-versa $i\in\overline{\Omega}$ and $j\in
\Omega$) for which there exist two indistinguishable finite state trajectories
of arbitrary length ending in $\left\{  i\right\}  $ and in $\left\{
j\right\}  $, respectively, both contained in $S^{\ast}$, such that the latter
is contained in $\overline{\Omega}$ (or vice-versa the former is contained in
$\overline{\Omega}$).
\end{itemize}

Let us give an example of the sets defined above. Because of the simplicity of
the example, such sets can be determined by inspection.

\begin{example}
\label{example_new}Consider the FSM depicted in Figure \ref{fig_new}, with
$X_{0}=\left\{  1,8\right\}  $ and $\Omega=\left\{  1,4\right\}  $%


\begin{figure}[ptb]
\begin{center}
\includegraphics[scale=0.5]{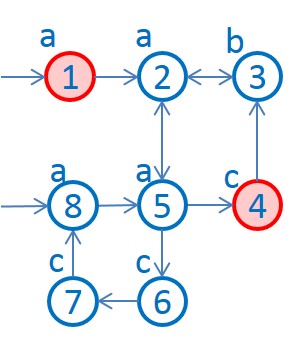}
\end{center}
\caption{FSM $M$ (Example \ref{example_new}).}%
\label{fig_new}%
\end{figure}

It is easily seen that:
\begin{align*}
S^{\ast}  &  =\left\{  \left(  1,8\right)  ,\left(  2,5\right)  ,\left(
4,6\right)  \right\}  ^{-}\cup\Theta\\
F^{\ast}  &  =\left\{  \left(  1,2\right)  ,\left(  1,5\right)  ,\left(
1,8\right)  ,\left(  2,5\right)  ,\left(  2,8\right)  ,\left(  5,8\right)
\right\}  ^{-}\cup\Theta\\
B^{\ast}\left(  S^{\ast}\right)   &  =\left\{  \left(  2,5\right)  ,\left(
4,6\right)  \right\}  ^{-}\cup\left(  \Theta\backslash\left\{  \left(
1,1\right)  \right\}  \right) \\
\Lambda^{\ast}  &  =\left\{  \left(  1,8\right)  \right\}  ^{-}\\
\Gamma^{\ast}  &  =\left\{  \left(  4,6\right)  \right\}  ^{-}%
\end{align*}

\end{example}

In the following subsections, we will formally define the above sets and give
algorithms for their computation. We will prove that the computation of these
sets has polynomial complexity in the state cardinality $\left\vert
X\right\vert $. In fact the sets $S^{\ast}$, $B^{\ast}\left(  \Sigma\right)
$, $F^{\ast}$, $\Gamma^{\ast}$ and $\Lambda^{\ast}$ are computed as fixed
points of appropriate recursions, whose convergence is assured after a number
of steps denoted by $s^{\ast}$, $b^{\ast}$, $f^{\ast}$, $g^{\ast}$ and
$l^{\ast}$, all upper bounded by $\left\vert X\right\vert ^{2}$.

\subsection{The set $S^{\ast}$}

\begin{definition}
\label{def_Sstar}The set $S^{\ast}$ is the maximal set of pairs $\left(
i,j\right)  \in\Pi$ such that there exist two indistinguishable state
executions $x_{1}\in\mathcal{X}^{\left\{  i\right\}  }\cap\mathcal{X}_{X_{0}}$
and $x_{2}\in\mathcal{X}^{\left\{  j\right\}  }\cap\mathcal{X}_{X_{0}}$.
\end{definition}

The pair of states in $S^{\ast}$ are indistinguishable in the sense of
\cite{Wang:2007} and of \cite{Sears:2014}, where algorithms were studied, in
the framework of Mealy automata with partially observable
transitions.\footnote{In \ \cite{Wang:2007} the relation between
indistinguishability and the observability notion introduced in
\cite{Lin:1988} is also established.}

In our framework, the set $S^{\ast}$ can be computed as follows.

Define the recursion, with $k=1,2,...$
\begin{align}
S_{1}  &  =\left(  X_{0}\times X_{0}\right)  \cap\Pi\nonumber\\
S_{k+1}  &  =\left\{  \left(  i,j\right)  \in\Pi:\left(  pre(i)\times
pre(j)\right)  \cap S_{k}\neq\emptyset\right\}  \cup S_{k} \label{recS}%
\end{align}

\begin{lemma}
\label{L_S}Consider the equation (\ref{recS}). Then,

$i)$ the least fixed point of the recursion, containing $\left(  X_{0}\times
X_{0}\right)  \cap\Pi$, exists, is unique and is equal to $S^{\ast}$;

$ii)$ the recursion reaches the fixed point $S^{\ast}$ in at most $s^{\ast
}<\left\vert X\right\vert ^{2}$ steps.
\end{lemma}

\begin{proof}
The set $\Pi$ is a fixed point of the recursion and the intersection of fixed
points is a fixed point. Therefore the least fixed point containing $\left(
X_{0}\times X_{0}\right)  \cap\Pi$ exists and is unique. Let $\widehat{S}$
denote such a fixed point. Then
\[
\widehat{S}=\left\{  \left(  i,j\right)  \in\Pi:\left(  pre(i)\times
pre(j)\right)  \cap\widehat{S}\neq\emptyset\right\}  \cup\widehat{S}%
\]
and hence $\left\{  \left(  i,j\right)  \in\Pi:\left(  pre(i)\times
pre(j)\right)  \cap\widehat{S}\neq\emptyset\right\}  \subset\widehat{S}$.
Suppose that $S_{k}\subset\widehat{S}$. Then $S_{k+1}$ is a subset of $\left(
\left\{  \left(  i,j\right)  \in\Pi:\left(  pre(i)\times pre(j)\right)
\cap\widehat{S}\neq\emptyset\right\}  \cup S_{k}\right)  \subset\widehat{S}$.
Since $S_{1}\subset\widehat{S}$, then, by induction, $S_{k}\subset\widehat{S}%
$, $\forall k=1,2,...$. If $S_{k+1}=S_{k}$ for some $k$ then $S_{k+i}=S_{k}$,
$\forall i\geq0$, and hence $S_{k}$ is a fixed point. But a finite $k$ such
that $S_{k+1}=S_{k}$ exists because of the finite cardinality of $\Pi$. Let
$\widehat{k}$ be the minimum value of $k$ such that $S_{k+1}=S_{k}$. It is
clear that $\widehat{k}$ is bounded by the number of not ordered pairs in
$\Pi$. Hence $\widehat{k}\leq\frac{\left\vert X\right\vert \left(  \left\vert
X\right\vert -1\right)  }{2}$. Therefore $S_{\widehat{k}}=\widehat{S}$. The
fact that $S^{\ast}=S_{\widehat{k}}$ comes from the maximality of $S^{\ast}$
(see Definition \ref{def_Sstar}). The statements $i)$ and $ii)$ are therefore proved.
\end{proof}

Let $n_{k}=\left\vert S_{k}\right\vert $, $p=\left\vert \Pi\right\vert $ and
$\nu=\max_{i\in N}\left\vert pre\left(  i\right)  \right\vert $. Then at step
$k+1$ the algorithm involves at most $\nu^{2}\left(  p-n_{k}\right)  n_{k}$
elementary computations, where an elementary computation is: given $\left(
i,j\right)  \in\Pi\backslash S_{k}$ and the pair $\left(  i^{\prime}%
,j^{\prime}\right)  \in pre\left(  i\right)  \times pre\left(  j\right)  $
check whether $\left(  i^{\prime},j^{\prime}\right)  \in S_{k}$. Since
$\nu^{2}\left(  p-n_{k}\right)  n_{k}\leq\nu^{2}\left\vert X\right\vert ^{4}$,
then the algorithm will stop after at most $2\nu^{2}\left\vert X\right\vert
^{4}\ln\left\vert X\right\vert $ elementary computations. Hence the spatial
complexity is $O\left(  \left\vert X\right\vert ^{2}\right)  $ and the time
complexity is $O\left(  \left\vert X\right\vert ^{5}\right)  $. Similar
observations on complexity hold for all the algorithms we will describe in the
following sections.

\subsection{The sets $F^{\ast}$ and $B^{\ast}\left(  \Sigma\right)  $}

\begin{definition}
The set $F^{\ast}$ is the maximal set of pairs $\left(  i,j\right)  \in\Pi$
which are $k-$forward indistinguishable, $\forall k\in\mathbb{Z}$, $k\geq1$.
\end{definition}

Define the recursion, with $k=1,2,...$,%
\begin{align}
F_{1}  &  =\Pi\nonumber\\
F_{k+1}  &  =\left\{  \left(  i,j\right)  \in F_{k}:\left(  succ(i)\times
succ(j)\right)  \cap F_{k}\neq\emptyset\right\}  \label{recF}%
\end{align}

\begin{lemma}
\label{LemmaF}Consider equation $\left(  \ref{recF}\right)  $. Then,

$i)$ $F_{k}$ is the set of all $k-$forward indistinguishable pairs;

$ii)$ the maximal fixed point of the recursion, contained in $\Pi$, is unique,
nonempty and is equal to $F^{\ast}$;

$iii)$ the recursion reaches its maximal fixed point $F^{\ast}$ in $f^{\ast
}<\left\vert X\right\vert ^{2}$ steps.
\end{lemma}

\begin{proof}
Statement $i)$ is true by definition of $k-$forward indistinguishable pairs.
Because of liveness assumption, the set $\Theta$ is a fixed point of the
recursion defined in equation (\ref{recF}), contained in $\Pi$. The union of
fixed points in $\Pi$ is a fixed point in $\Pi$ and therefore the maximal
fixed point of the recursion, contained in $\Pi$, is unique and nonempty. Let
$\widehat{F}$ be such fixed point. Then $\forall\left(  i,j\right)
\in\widehat{F}$, $\left(  succ(i)\times succ(j)\right)  \cap\widehat{F}%
\neq\emptyset$. Let us suppose that $\widehat{F}\subset F_{k}$. Then
$\widehat{F}\subset F_{k+1}$. Since $\widehat{F}\subset F_{1}$, then, by
induction, $\widehat{F}\subset F_{k}$, $\forall k=1,2,...$. Moreover
$F_{k+1}\subset F_{k}$, $\forall k=1,2,...$. If $F_{k+1}\subset F_{k}$ for
some $k$ then $F_{k+i}=F_{k}$, $\forall i\geq0$, and hence $F_{k}$ is a fixed
point. But a finite $k$ such that $F_{k+1}=F_{k}$ exists because of the finite
cardinality of $\Pi$. Let $\widehat{k}$ be the minimum value of $k$ such that
$F_{k+1}=F_{k}$. Then $\widehat{F}\subset F_{\widehat{k}}\subset\widehat{F}$
and hence $F_{\widehat{k}}=\widehat{F}$. It is clear that $\widehat{k}$ is
bounded by the number of not ordered pairs in $\Pi$. Hence $\widehat{k}%
\leq\frac{\left\vert X\right\vert \left(  \left\vert X\right\vert -1\right)
}{2}$. The fact that $F^{\ast}=F_{\widehat{k}}$ comes from the definition of
the set $F^{\ast}$. The statements $ii)$ and $iii)$ are therefore proved.
\end{proof}

In a similar way, given $\Sigma\subset\Pi$, we can introduce the following

\begin{definition}
The set $B^{\ast}\left(  \Sigma\right)  $ is the maximal set of pairs $\left(
i,j\right)  \in\Sigma$ which are $k-$backward indistinguishable in $\Sigma$,
$\forall k\in\mathbb{Z}$, $k\geq1$.
\end{definition}

Define the recursion, with $k=1,2,...$,
\begin{align}
B_{1}(\Sigma)  &  =\Sigma\label{recB}\\
B_{k+1}(\Sigma)  &  =\left\{  \left(  i,j\right)  \in B_{k}(\Sigma):\left(
pre(i)\times pre(j)\right)  \cap B_{k}(\Sigma)\neq\emptyset\right\} \nonumber
\end{align}

\begin{lemma}
\label{LemmaB}Consider equation $\left(  \ref{recB}\right)  $. Then,

$i)$ $B_{k}(\Sigma)$ is the set of all $k-$backward indistinguishable pairs in
$\Sigma$;

$ii)$ if $B^{\ast}\left(  \Sigma\right)  \neq\emptyset$, then the maximal
fixed point of the recursion (\ref{recB}), contained in $\Sigma$, is unique,
nonempty and is equal to $B^{\ast}\left(  \Sigma\right)  $. Otherwise $\exists
k<\left\vert X\right\vert ^{2}$ such that $B_{k}(\Sigma)=\emptyset$;

$iii)$ If $B^{\ast}\left(  \Sigma\right)  \neq\emptyset$, the recursion
reaches its maximal fixed point in $b^{\ast}<\left\vert X\right\vert ^{2}$ steps.
\end{lemma}

\begin{proof}
Statement $i)$ is true by definition of $k-$backward indistinguishable pairs,
since $\Sigma\subset\Pi$. $ii)$ Suppose that $B^{\ast}\left(  \Sigma\right)
\neq\emptyset$. Then $B^{\ast}\left(  \Sigma\right)  \subset B_{k}(\Sigma)$,
$\forall k=1,2,...$. and is a fixed point of the recursion defined in equation
(\ref{recB}). The union of fixed points in $\Sigma$ is a fixed point and
therefore the maximal fixed point of the recursion, contained in $\Sigma$, is
nonempty and is unique. Let $\widehat{B}$ be such fixed point. Then
$\forall\left(  i,j\right)  \in\widehat{B}$, $\left(  pre(i)\times
pre(j)\right)  \cap\widehat{B}\neq\emptyset$. Let us suppose that
$\widehat{B}\subset B_{k}(\Sigma)$. Then $\widehat{B}\subset B_{k+1}(\Sigma)$.
Since $\widehat{B}\subset B_{1}(\Sigma)$, then, by induction, $\widehat{B}%
\subset B_{k}(\Sigma)$, $\forall k=1,2,...$. Moreover $B_{k+1}(\Sigma)\subset
B_{k}(\Sigma)$, $\forall k=1,2,...$. If $B_{k+1}(\Sigma)\subset B_{k}(\Sigma)$
for some $k$ then $B_{k+i}(\Sigma)=B_{k}(\Sigma)$, $\forall i\geq0$, and hence
$B_{k}(\Sigma)$ is a fixed point. But a finite $k$ such that $B_{k+1}%
(\Sigma)=B_{k}(\Sigma)$ exists because of the finite cardinality of $\Pi$. Let
$\widehat{k}$ be the minimum value of $k$ such that $B_{k+1}(\Sigma
)=B_{k}(\Sigma)$. Since $\widehat{B}\subset B_{\widehat{k}}(\Sigma
)\subset\widehat{B}$, then $B_{\widehat{k}}(\Sigma)=\widehat{B}$. It is clear
that $\widehat{k}$ is bounded by the number of not ordered pairs in $\Pi$.
Hence $\widehat{k}\leq\frac{\left\vert X\right\vert \left(  \left\vert
X\right\vert -1\right)  }{2}$. The fact that $B^{\ast}\left(  \Sigma\right)
=B_{b^{\ast}}(\Sigma)$ comes from the definition of the set $B^{\ast}\left(
\Sigma\right)  $. If $B^{\ast}\left(  \Sigma\right)  =\emptyset$, then by
definition of the recursion (\ref{recB}) there exists $k$ such that
$B_{k}(\Sigma)=\emptyset$. The statements $ii)$ and $iii)$ are therefore proved.
\end{proof}

\subsection{The sets $\Lambda^{\ast}$ and $\Gamma^{\ast}$}

Given $S^{\ast}$ and $\Omega\subset X$, we now define the sets $\Lambda_{k}$
and $\Lambda^{\ast}$ that are subsets of $F_{k}$ and $F^{\ast}$, i.e. the sets
of forward indistinguishable pairs, for finite and infinite steps, respectively.

\begin{definition}
$\Lambda_{k}$ is the set of pairs $\left(  i,j\right)  \in S^{\ast}$, with
$i\in\Omega$ and $j\in\overline{\Omega}$ (or vice-versa $i\in\overline{\Omega
}$ and $j\in\Omega$) for which there exist two indistinguishable executions
$x_{1}\in\mathcal{X}_{\left\{  i\right\}  }$ and $x_{2}\in\mathcal{X}%
_{\left\{  j\right\}  }$, $\left\vert x_{1}\right\vert =\left\vert
x_{2}\right\vert =k$, such that $x_{2}\left(  h\right)  \in\overline{\Omega}$,
$\forall h\in\left[  1,k\right]  $ ($x_{1}\left(  h\right)  \in\overline
{\Omega}$, $\forall h\in\left[  1,k\right]  $, respectively). $\Lambda^{\ast}$
is the set of pairs $\left(  i,j\right)  \in S^{\ast}$ such that
\[
\forall\overline{k}\in\mathbb{Z},\exists k\geq\overline{k}:\left(  i,j\right)
\in\Lambda_{k}%
\]

\end{definition}

The sets $\Lambda_{k}$ and $\Lambda^{\ast}$ can be computed by defining the
recursion, $k=1,2,...$
\begin{align}
\Psi_{1}  &  =\left(  X\times\overline{\Omega}\right)  \cap S^{\ast
}\nonumber\\
\Psi_{k+1}  &  =\left\{  \left(  i,j\right)  \in\Psi_{k}:\left(  succ(i)\times
succ(j)\right)  \cap\Psi_{k}\neq\emptyset\right\}  \label{rec_Lambda}%
\end{align}

In fact, we can prove the following:

\begin{lemma}
\label{L_Lambda}Consider equation (\ref{rec_Lambda}). Then,

$i)$ $\Lambda_{k}=\left(  \Psi_{k}\cap\left(  \Omega\times\overline{\Omega
}\right)  \right)  ^{-}$;

$ii)$ If $\Psi_{k}\neq\emptyset$, $\forall k=1,2,...$, the maximal fixed point
$\Psi^{\ast}$ of the recursion defined in $\left(  \ref{rec_Lambda}\right)  $,
contained in $X\times\overline{\Omega}$, is nonempty and unique. Otherwise
$\exists k<\left\vert X\right\vert ^{2}$ such that $\Psi_{k}=\emptyset$ and
$\Psi^{\ast}=\emptyset$;

$iii)$ If $\Psi^{\ast}\neq\emptyset$ the recursion defined in $\left(
\ref{rec_Lambda}\right)  $ reaches this maximal fixed point in $l^{\ast
}<\left\vert X\right\vert ^{2}$ steps;

$iv)$ $\Lambda^{\ast}=\left(  \Psi^{\ast}\cap\left(  \Omega\times
\overline{\Omega}\right)  \right)  ^{-}$.
\end{lemma}

\begin{proof}
The recursion defined in equation (\ref{rec_Lambda}), up to the
initialization, is identical to recursion defined in (\ref{recF}). Therefore
$\Psi_{k}$ is the set of $k-$forward indistinguishable pairs $\left(
i,j\right)  $ for which there exists two indistinguishable state trajectories
$x_{1}\in\mathcal{X}_{(i)}$ and $x_{2}\in\mathcal{X}_{(j)}$, with $\left\vert
x_{1}\right\vert =\left\vert x_{2}\right\vert =k$, such that $x_{2}\left(
h\right)  \in\overline{\Omega}$, $\forall h=1...k$. Therefore statement $i)$
is true by definition of $\Lambda_{k}$. By using the same arguments as in the
proof of Lemma \ref{LemmaF}, the maximal fixed point $\Psi^{\ast}$ of the
recursion \ref{rec_Lambda}, contained in $X\times\overline{\Omega}$ is unique.
However it could be equal to the emptyset. If $\Psi^{\ast}\neq\emptyset$, then
again by using the same arguments as in the proof of Lemma \ref{LemmaF}, there
exists $\widehat{k}<\left\vert X\right\vert ^{2}$ such that $\Psi
_{\widehat{k}+1}=\Psi_{\widehat{k}}$ and hence statements $ii)$ and $iii)$
hold. If $\Psi^{\ast}=\emptyset$, then $\Psi_{\widetilde{k}}=\emptyset$, for
some $\widetilde{k}<\left\vert X\right\vert ^{2}$, and again statements $ii)$
and $iii)$ hold. The last statement comes from the definition of
$\Lambda^{\ast}$.
\end{proof}

It can be easily verified that:
\begin{align}
\Lambda_{1}  &  =\left(  \left(  \Omega\times\overline{\Omega}\right)
\cup\left(  \Omega\times\overline{\Omega}\right)  \right)  \cap S^{\ast
}\nonumber\\
\Lambda_{k+1}  &  \subset\Lambda_{k}\subset\left(  F_{k}\cap S^{\ast}\right)
\nonumber\\
\Lambda^{\ast}  &  =\bigcap\limits_{k\in\mathbb{Z}}\Lambda_{k}\subset\left(
F^{\ast}\cap S^{\ast}\right)  \label{g}%
\end{align}

The sets $\Gamma_{k}$ and $\Gamma^{\ast}$ take into account the "backward
executions" of the FSM. In fact they are subsets of $B_{k}\left(  S^{\ast
}\right)  $ and of $B^{\ast}\left(  S^{\ast}\right)  $, respectively, and are
defined as follows:

\begin{definition}
$\Gamma_{k}$ is the set of pairs $\left(  i,j\right)  \in S^{\ast}$, with
$i\in\Omega$ and $j\in\overline{\Omega}$ (or vice-versa $i\in\overline{\Omega
}$ and $j\in\Omega$) for which there exist two indistinguishable executions
$x_{1}\in\mathcal{X}^{\left\{  i\right\}  }$ and $x_{2}\in\mathcal{X}%
^{\left\{  j\right\}  }$, $\left\vert x_{1}\right\vert =\left\vert
x_{2}\right\vert =k$, such that $\left(  x_{1}\left(  h\right)  ,x_{2}\left(
h\right)  \right)  \in S^{\ast}\cap\left(  X\times\overline{\Omega}\right)  $,
$\forall h\in\left[  1,k\right]  $ (vice-versa $\left(  x_{1}\left(  h\right)
,x_{2}\left(  h\right)  \right)  \in S^{\ast}\cap\left(  \overline{\Omega
}\times X\right)  $ $\forall h\in\left[  1,k\right]  $, respectively).
$\Gamma^{\ast}$ is the set of pairs $\left(  i,j\right)  \in S^{\ast}$ such
that
\[
\forall\overline{k}\in\mathbb{Z},\exists k\geq\overline{k}:\left(  i,j\right)
\in\Gamma_{k}%
\]

\end{definition}

Define the recursion, with $k=1,2,...$
\begin{align}
\Xi_{1}  &  =\left(  X\times\overline{\Omega}\right)  \cap S^{\ast
}\label{rec_Gamma}\\
\Xi_{k+1}  &  =\left\{  \left(  i,j\right)  \in\Xi_{k}:\left(  prec(i)\times
prec(j)\right)  \cap\Xi_{k}\neq\emptyset\right\} \nonumber
\end{align}

\begin{lemma}
\label{L_Gamma}Consider equation (\ref{rec_Gamma}). Then,

$i)$ $\Gamma_{k}=\left(  \Xi_{k}\cap\left(  \Omega\times\overline{\Omega
}\right)  \right)  ^{-}$

$ii)$ If $\Xi_{k}\neq\emptyset$ $\forall k$, then the maximal fixed point
$\Xi^{\ast}$ of the recursion, contained in $\left(  X\times\overline{\Omega
}\right)  \cap S^{\ast}$, is unique and nonempty. Otherwise $\exists
k<\left\vert X\right\vert ^{2}$ such that $\Xi_{k}=\emptyset$ and $\Xi^{\ast
}=\emptyset$.

$iii)$ If $\Xi^{\ast}\neq\emptyset$, the recursion reaches this maximal
\ fixed point in $g^{\ast}<\left\vert X\right\vert ^{2}$ steps

$iv)$ $\Gamma^{\ast}=\left(  \Xi^{\ast}\cap\left(  \Omega\times\overline
{\Omega}\right)  \right)  ^{-}$.
\end{lemma}

\begin{proof}
The recursion defined in equation (\ref{rec_Gamma}), up to the initialization,
is identical to the recursion defined in (\ref{recB}). Therefore $\Xi_{k}$ is
the set of $k-$backward indistinguishable pairs $\left(  i,j\right)  $ for
which there exists two indistinguishable state trajectories $x_{1}%
\in\mathcal{X}^{(i)}$ and $x_{2}\in\mathcal{X}^{(j)}$, with $\left\vert
x_{1}\right\vert =\left\vert x_{2}\right\vert =k$, such that $x_{2}\left(
h\right)  \in\overline{\Omega}$, $\forall h=1...k$. Therefore statement $i)$
is true by definition of $\Gamma_{k}$. By using the same arguments as in the
proof of Lemma \ref{LemmaB}, the maximal fixed point $\Xi^{\ast}$ of the
recursion (\ref{rec_Gamma}), contained in $\left(  X\times\overline{\Omega
}\right)  \cap S^{\ast}$ is unique. However it could be equal to the emptyset.
If $\Xi^{\ast}\neq\emptyset$, then again by using the same arguments as in the
proof of Lemma \ref{LemmaB}, there exists $\widehat{k}<\left\vert X\right\vert
^{2}$ such that $\Xi_{\widehat{k}+1}=\Xi_{\widehat{k}}$ and hence statements
$ii)$ and $iii)$ hold. If $\Xi^{\ast}=\emptyset$, then $\Xi_{\widetilde{k}%
}=\emptyset$, for some $\widetilde{k}<\left\vert X\right\vert ^{2}$, and
therefore statements $ii)$ and $iii)$ hold. The last statement comes from the
definition of $\Gamma^{\ast}$.
\end{proof}

It can be easily verified that:
\begin{align}
\Gamma_{1}  &  =\left(  \left(  \Omega\times\overline{\Omega}\right)
\cup\left(  \Omega\times\overline{\Omega}\right)  \right)  \cap S^{\ast
}\nonumber\\
\Gamma_{k+1}  &  \subset\Gamma_{k}\subset B_{k}\left(  S^{\ast}\right)
\nonumber\\
\Gamma^{\ast}  &  =\bigcap\limits_{k\in\mathbb{Z}}\Gamma_{k}\subset B^{\ast
}\left(  S^{\ast}\right)  \subset S^{\ast} \label{f}%
\end{align}

\section{\label{sec3}Main results}

In this section we characterize the properties introduced in Definitions
\ref{Def_somehow}, \ref{def_oneshot}, \ref{Def_eventual} and
\ref{Def_critical}. We first derive necessary and sufficient conditions for
each of those properties to hold. Then, some equivalent conditions are given
in terms of simple set inclusions depending on the existence of some suitable
parameters. The values of these parameters allow the computation of an upper
bound for the delay of the diagnosis, and of a lower bound for the uncertainty
radius of the diagnosis.

Given the FSM $M=(X,X_{0},Y,H,\Delta)$, define the FSM $\widetilde{M}%
=(X,X_{0},Y,H,\widetilde{\Delta})$, where $\left(  i,j\right)  \in
\widetilde{\Delta}$ if and only if $\left(  i,j\right)  \in\Delta$ and
$i\notin\Omega$. Let $\widetilde{S}^{\ast}$ be the set of pairs reachable from
$X_{0}$ with two indistinguishable state evolutions, computed for
$\widetilde{M}$. Obviously $\widetilde{S}^{\ast}\subset S^{\ast}$.

The FSM $\widetilde{M}$ is not alive in general. However, as pointed out in
Section \ref{sec2}, the algorithm for the computation of $S^{\ast}$ does not
depend on the liveness assumption, so the set \ $\widetilde{S}^{\ast}$ can be
computed by means of the same algorithm as $S^{\ast}$. As an example, consider
the FSM represented in Figure \ref{fig_nodiag}. The FSM $\widetilde{M}$ is
represented in Figure \ref{tilde_nodiag}.%


\begin{figure}[ptb]
\begin{center}
\includegraphics[scale=0.5]{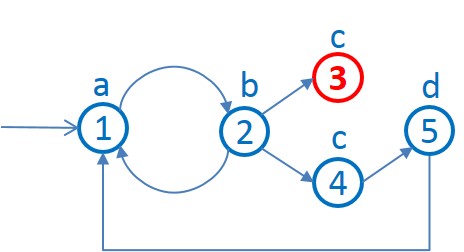}
\end{center}
\caption{The FSM\ $\protect\widetilde{M}$, corresponding to FSM\ $M$ in Fig. \ref{fig_nodiag}.}%
\label{tilde_nodiag}%
\end{figure}

The finite nonnegative values $b^{\ast}$, $\widetilde{b}^{\ast}$, $f^{\ast}%
$,$g^{\ast}$ and $l^{\ast}$ are well defined:
\begin{align*}
b^{\ast}  &  =\min b:B^{\ast}(S^{\ast})=B_{b}(S^{\ast})\\
\widetilde{b}^{\ast}  &  =\min b:B^{\ast}(\widetilde{S}^{\ast})=B_{b}%
(\widetilde{S}^{\ast})\\
f^{\ast}  &  =\min f:F^{\ast}=F_{f}\\
g^{\ast}  &  =\min g:\Gamma^{\ast}=\Gamma_{g}\\
l^{\ast}  &  =\min l:\Lambda^{\ast}=\Lambda_{l}%
\end{align*}

\subsection{Parametric $\Omega-$ diagnosability\label{sub1}}

Consider the set $B^{\ast}\left(  \widetilde{S}^{\ast}\right)  \cap
\Lambda^{\ast}$. On the basis of the definition of the sets given in the
previous section, a pair of states $\left(  i,j\right)  $ in the set $B^{\ast
}\left(  \widetilde{S}^{\ast}\right)  \cap\Lambda^{\ast}$ is such that only
one of the two states $i$ or $j$ belongs to $\Omega$. Such\ $i$ and $j$ are
the ending states of a pair of arbitrarily long indistinguishable state
executions of the system $\widetilde{M}$ (which are also executions of $M$)
with initial states in $X_{0}$, with one of these executions which never
crosses the set $\Omega$. Moreover, $i$ and $j$ are the initial states of a
pair of arbitrarily long indistinguishable state executions of the system $M$,
with one of these executions which never crosses the set $\Omega$. Therefore,
given these executions, however long the transient and the delay are, and
however loose is the required accuracy, it will not be possible to decide
whether the critical set $\Omega$ has been crossed or not. As a consequence,
bearing in mind the definition of $\widetilde{M}$ and the equivalence between
parametric $\Omega-diag$ and parametric $\Omega-diag$ with $T=0$, we can
establish the following necessary and sufficient condition for an FSM to be
parametrically $\Omega-diag$.

\begin{theorem}
\label{general_as}$M$ is parametrically $\Omega-diag$ if and only if \
\begin{equation}
B^{\ast}\left(  \widetilde{S}^{\ast}\right)  \cap\Lambda^{\ast}=\emptyset
\label{gen_exist}%
\end{equation}

\end{theorem}

\begin{proof}
\textbf{Sufficiency: }By definition of $\widetilde{M}$ and of $\widetilde{S}%
^{\ast}$, the set
\[
\left\{  i\in\Omega:\left(  i,j\right)  \in B^{\ast}\left(  \widetilde{S}%
^{\ast}\right)  ,i\neq j\right\}
\]
describes the set of all states $i$ in $\Omega$, such that for any
$k\geq\widetilde{b}^{\ast}$ there exists a state execution $x\in
\mathcal{X}_{X_{0}}$, with finite $k_{x}\geq k$, $x(k_{x})=i$, but the value
of the state $x(k_{x})$ cannot be reconstructed, knowing the output evolution
up to step $k_{x}$. If $B^{\ast}\left(  \widetilde{S}^{\ast}\right)
=\emptyset$, then there exists $k^{\prime}\geq\widetilde{b}^{\ast}$ such that
any execution in $\mathcal{X}$ is such that either $k_{x}=\infty$ or
$k_{x}<k^{\prime}$. But $B^{\ast}\left(  \widetilde{S}^{\ast}\right)
=B_{\widetilde{b}^{\ast}}(\widetilde{S}^{\ast})$, and therefore $k^{\prime
}=\widetilde{b}^{\ast}$. Hence the condition in Definition \ref{Def_somehow}
is satisfied with $T=0$, $\tau=\widetilde{b}^{\ast}$ and $\delta=0$. By
definition of $\Lambda^{\ast}$, if $\Lambda^{\ast}=\emptyset$ then given a
pair $\left(  i,j\right)  \in\Omega\times\overline{\Omega}$, any state
execution $x\in\mathcal{X}_{\left\{  j\right\}  ,\infty}$ is such that
$x(h)\in\Omega$, for some $h\in\left[  1,l^{\ast}\right]  $. Therefore if
$\Lambda^{\ast}=\emptyset$ the condition in Definition \ref{Def_somehow} is
satisfied with $T=0$, $\tau=0$, $\delta=l^{\ast}-1$, $\gamma_{1}=\gamma
_{2}=l^{\ast}-1$. Finally, if $B^{\ast}\left(  \widetilde{S}^{\ast}\right)
\cap\Lambda^{\ast}=\emptyset$, then the condition in Definition
\ref{Def_somehow} is satisfied with $T=0$, $\tau=\widetilde{b}^{\ast}$,
$\delta=l^{\ast}-1$, $\gamma_{1}=\gamma_{2}=\widetilde{l}^{\ast}-1$ and the
proof of the sufficiency is complete. \textbf{Necessity}: By Definition
\ref{Def_somehow} $M$ is parametrically $\Omega-diag$ only if it is so with
$T=0$. Suppose that $B^{\ast}\left(  \widetilde{S}^{\ast}\right)  \cap
\Lambda^{\ast}\neq\emptyset$. Then since $\Lambda^{\ast}\subset F^{\ast}$,
$B^{\ast}\left(  \widetilde{S}^{\ast}\right)  \cap F^{\ast}$ is not a subset
of $\overline{\Lambda^{\ast}}$. Hence there exists $\left(  i,j\right)  \in
B^{\ast}\left(  \widetilde{S}^{\ast}\right)  \cap F^{\ast}$ which belongs to
$\Lambda^{\ast}$. Then, by definition of $B^{\ast}\left(  \widetilde{S}^{\ast
}\right)  $ and of $\Lambda^{\ast}$ it is not possible to decide about the
crossing of the set $\Omega$ with any delay. Hence $M$ is not parametrically
$\Omega-diag$, with $T=0$. Hence it is not parametrically $\Omega-diag$.
\end{proof}

As a consequence of the result above, the following equivalent condition can
be obtained

\begin{corollary}
\label{th_general}If there exist $b\in\left[  1,\widetilde{b}^{\ast}\right]
$, $f\in\left[  1,f^{\ast}\right]  $ and $l\in\left[  1,l^{\ast}\right]  $
such that \
\begin{equation}
\left(  B_{b}\left(  \widetilde{S}^{\ast}\right)  \cap F_{f}\right)
\subset\overline{\Lambda_{l}} \label{general}%
\end{equation}
then $M$ is parametrically $\Omega-diag$ with $T=0$, $\tau=b-1$, $\delta
=\max\left\{  f,l\right\}  -1$, $\gamma_{1}=\gamma_{2}=l-1$. Conversely, if
$M$ is parametrically $\Omega-diag$, then there exist $b\in\left[
1,\widetilde{b}^{\ast}\right]  $, $f\in\left[  1,f^{\ast}\right]  $ and
$l\in\left[  1,l^{\ast}\right]  $ such that inclusion (\ref{general}) holds.
\end{corollary}

\begin{proof}
Since $\Lambda^{\ast}\subset\left(  F^{\ast}\cap S^{\ast}\right)  $, then
$B^{\ast}\left(  \widetilde{S}^{\ast}\right)  \cap\Lambda^{\ast}=\emptyset$ if
and only if $B^{\ast}\left(  \widetilde{S}^{\ast}\right)  \cap\Lambda^{\ast
}\cap F^{\ast}\cap S^{\ast}=\emptyset$, which is equivalent to write $B^{\ast
}\left(  \widetilde{S}^{\ast}\right)  \cap F^{\ast}\cap S^{\ast}%
\subset\overline{\Lambda^{\ast}}$. Since $\widetilde{S}^{\ast}\subset S^{\ast
}$, then the condition \ref{gen_exist} is equivalent to the inclusion%
\[
\left(  B^{\ast}\left(  \widetilde{S}^{\ast}\right)  \cap F^{\ast}\right)
\subset\overline{\Lambda^{\ast}}%
\]
Moreover, $\left(  B^{\ast}\left(  \widetilde{S}^{\ast}\right)  \cap F^{\ast
}\right)  \subset\left(  B_{b}\left(  \widetilde{S}^{\ast}\right)  \cap
F_{f}\right)  $, $\forall b\in\left[  1,\widetilde{b}^{\ast}\right]  $,
$\forall f\in\left[  1,f^{\ast}\right]  $ and $\Lambda^{\ast}\subset
\Lambda_{l}$, $\forall l\in\left[  1,l^{\ast}\right]  $. Therefore, if for
some $b\in\left[  1,\widetilde{b}^{\ast}\right]  $, for some $f\in\left[
1,f^{\ast}\right]  $ and for some $l\in\left[  1,l^{\ast}\right]  $ the
inclusion \ref{general} holds, then we can write%
\[
\left(  B^{\ast}\left(  \widetilde{S}^{\ast}\right)  \cap F^{\ast}\right)
\subset\left(  B_{b}\left(  \widetilde{S}^{\ast}\right)  \cap F_{f}\right)
\subset\overline{\Lambda_{l}}\subset\overline{\Lambda^{\ast}}%
\]
and hence the first statement comes from the sufficient part of Theorem
\ref{general_as}. The evaluation of the parameters $T$, $\tau$, $\delta$,
$\gamma_{1}$ and $\gamma_{2}$ is obvious. The second statement is
straightforward from the necessity part of the same Theorem \ref{general_as}.
\end{proof}

In the statement of the previous corollary, the parameters $b$, $f$ and $l$
appear explicitly. Note that $\Lambda^{\ast}\subset F^{\ast}$ but in general
$\Lambda_{l}$ is not a subset of $F_{f}$, for $l\neq l^{\ast}$ or $f\neq
f^{\ast}$. Therefore the inclusion $\left(  \ref{general}\right)  $ defines
the set of all the values $b$, $f$ and $l$ such that $M$ is parametrically
$\Omega-diag$. Given this set, an upper bound for the delay and a lower bound
for the uncertainty radius can be evaluated. Let us explain how this can be
done, in particular how the condition $\left(  \ref{general}\right)  $ allows
the determination of the delay of the diagnosis of the crossing event, of the
uncertainty about the time at which the event occurred and of the duration of
the transient where the diagnosis is not possible or not required. This
description gives also the tools for the design of the online diagnoser.

Suppose $\left(  \ref{general}\right)  $ holds for some $b$, $f$ and $l$.
Given an infinite execution $x$ and the output string up to current step
$k\geq b+\max\left\{  f,l\right\}  -1$, let $\widehat{x}\left(  k\right)
\in2^{X}$ be the set of discrete states at step $k-(\max\left\{  f,l\right\}
-1)$\ that are compatible with the observations up to step $k$. Suppose that
$k_{x}\geq b$. Then at $k=k_{x}+\max\left\{  f,l\right\}  -1$, $\widehat{x}%
\left(  k\right)  \cap\Omega\neq\emptyset$. If $\widehat{x}\left(  k\right)
\subset\Omega$, then we can deduce that the set $\Omega$ was crossed at step
$k_{x}=k-(\max\left\{  f,l\right\}  -1)$. Otherwise suppose that
$\widehat{x}\left(  k\right)  =\left\{  i,j,h\right\}  $, with only the state
$i$ belonging to $\Omega$. Then each pair $\left(  i,j\right)  $, $\left(
i,h\right)  $ and $\left(  j,h\right)  $ belongs to $\left(  B_{b}\left(
\widetilde{S}^{\ast}\right)  \cap F_{f}\right)  $, because they have not been
distinguished at step $k$. Since the inclusion $\left(  \ref{general}\right)
$ holds, then at step $k$ we are sure that the actual execution $x$ of $M$ is
such that $x(h)\in\Omega$, for some $h\in\left[  k_{x},k_{x}+l-1\right]  $.
Suppose that $b>1$ and $k_{x}\leq b-1$. Since $B_{k_{x}}\left(  \widetilde{S}%
^{\ast}\right)  $ is not in general a subset of $B_{b}\left(  \widetilde{S}%
^{\ast}\right)  $, then condition $\left(  \ref{general}\right)  $ does not
allow the detection of the crossing event, based on the information available
at step $k_{x}+\max\left\{  f,l\right\}  -1$. Hence detection of the first
crossing event occurs with a maximum delay $\delta=\max\left\{  f,l\right\}
-1$, with uncertainty $\gamma=l-1$ and with $\tau=b-1$. \ Therefore the system
is parametrically diagnosable with $T=0$, $\delta=\max\left\{  f,l\right\}
-1$, $\gamma=l-1$ and $\tau=b-1$. Finally, since $B_{1}\left(  \widetilde{S}%
^{\ast}\right)  =\widetilde{S}^{\ast}$, $F_{1}=\Pi$, $\Lambda_{1}=\left(
\left(  \Omega\times\overline{\Omega}\right)  \cup\left(  \overline{\Omega
}\times\Omega\right)  \right)  $, then if $\left(  \ref{general}\right)  $
holds in the very special case of $b=f=l=1$, then $\widetilde{S}^{\ast}%
\subset\left(  \left(  \Omega\times\Omega\right)  \cup\left(  \overline
{\Omega}\times\overline{\Omega}\right)  \right)  $, and detection occurs with
a maximum delay $\delta=0$, with uncertainty $\gamma=0$ and with $\tau=0$.

\subsection{\label{omega_diag}$\Omega-$diagnosability}

Consider now $\Omega-$diagnosability as defined in Definition
\ref{def_oneshot}.

Consider the set $\widetilde{S}^{\ast}\cap\Lambda^{\ast}$. By similar
reasoning as in the previous subsection, the set $\widetilde{S}^{\ast}%
\cap\Lambda^{\ast}$ is the set of pairs $\left(  i,j\right)  $, where only one
of the two states $i$ or $j$ belongs to $\Omega$, which are the ending states
of a pair of indistinguishable state executions of the system $\widetilde{M}$,
with initial state in $X_{0}$, such that one of these executions never crosses
the set $\Omega$, and are the initial states of a pair of arbitrarily long
indistinguishable state executions of the system $M$, such that one of them
never crosses $\Omega$. Therefore, recalling the definition of the system
$\widetilde{M}$, we can prove the following:

\begin{theorem}
\label{Th_LaF_exist}$M$ is $\Omega-diag$ if and only if \
\begin{equation}
\widetilde{S}^{\ast}\cap\Lambda^{\ast}=\emptyset\label{asymp_oneshot}%
\end{equation}

\end{theorem}

\begin{proof}
\textbf{Sufficiency:} Setting $b=1$ in condition (\ref{general}),
$B_{1}\left(  \widetilde{S}^{\ast}\right)  =\widetilde{S}^{\ast}$ and we
obtain condition (\ref{asymp_oneshot}). Hence $\tau=0$ and $M$ is
$\Omega-diag$. \ \textbf{Necessity}: if $\widetilde{S}^{\ast}\cap\Lambda
^{\ast}\neq\emptyset$ then for any $f,l\in\mathbb{Z}$, and since
$\Lambda^{\ast}\subset F^{\ast}$ there exists $\left(  i,j\right)
\in\widetilde{S}^{\ast}\cap F^{\ast}$, such that $\left(  i,j\right)
\in\Lambda_{l}$. Then there exists $x\in\mathcal{X}$ such that $x(k)=i$ (or
$j$), and the pair $\left(  i,j\right)  $ cannot be distinguished at step
$k+f$ from $\mathbf{y}\left(  \left.  x\right\vert _{\left[  1,k+f\right]
}\right)  $, $\forall f\in\mathbb{Z}$. Since $\left(  i,j\right)  \in
\Lambda^{\ast}$ there exists a pair $\left(  x_{1},x_{2}\right)  $ of infinite
indistinguishable evolutions starting from $\left(  i,j\right)  $, with the
property that only one of them crosses the set $\Omega$. Therefore there does
not exist $\delta$ such that at step $k+\delta$ it is possible to decide if a
crossing event occurred in the interval $\left[  1,k+\delta\right]  $. Hence
the given condition is necessary.
\end{proof}

Condition $\left(  \ref{asymp_oneshot}\right)  $ implies diagnosability as
defined in \cite{Sampath95}. The proof of the necessity in Theorem
\ref{Th_LaF_exist} above shows that condition $\left(  \ref{asymp_oneshot}
\right)  $ is necessary also for the property of \cite{Sampath95} to hold.
Hence, the diagnosability property of Definition \ref{def_oneshot} and the one
defined in \cite{Sampath95} are equivalent.

As in the previous subsection, we obtain the following equivalent condition
expressed in terms of the parameters for which $\Omega-$diagnosability holds.

\begin{corollary}
\label{Th_LaF_finite}The FSM $M$ is $\Omega-diag$ with delay $\delta$ and
uncertainty radius $\gamma$ if
\begin{equation}
\left(  \widetilde{S}^{\ast}\cap F_{f}\right)  \subset\overline{\Lambda_{l}}
\label{finite_oneshot}%
\end{equation}
where $f\leq\delta+1$, $l\leq\delta+1$ and $l\leq\gamma+1$. Conversely, if $M$
is $\Omega-diag$, then there exist $f\in\left[  1,f^{\ast}-1\right]  $ and
$l\in\left[  1,l^{\ast}-1\right]  $ such that inclusion (\ref{finite_oneshot}) holds.
\end{corollary}

\begin{proof}
Straightforward consequence of Theorem \ref{Th_LaF_exist}.
\end{proof}

Condition $\left(  \ref{finite_oneshot}\right)  $ gives the tools for the
computation of the delay between the occurrence of the critical event and its
detection. More precisely, let us explain how the condition $\left(
\ref{finite_oneshot}\right)  $ allows the determination of the delay of the
diagnosis of the crossing event and the uncertainty about the time at which
the event occurred, and how the online detection can be done.

Suppose $\left(  \ref{finite_oneshot}\right)  $ holds for some $f$ and $l$.
Let $k^{\prime}$ be the first $k\geq\max\left\{  f,l\right\}  $ such that
$\widehat{x}\left(  k\right)  \cap\Omega\neq\emptyset$. If $\widehat{x}\left(
k^{\prime}\right)  \subset\Omega$, then we can deduce that the set $\Omega$
was crossed for the first time at step $k^{\prime}-(\max\left\{  f,l\right\}
-1)$. Otherwise suppose that $\widehat{x}\left(  k^{\prime}\right)  =\left\{
i,j,h\right\}  $, with only the state $i$ belonging to $\Omega$. Then each
pair $\left(  i,j\right)  $, $\left(  i,h\right)  $ and $\left(  j,h\right)  $
belongs to $\widetilde{S}^{\ast}\cap F_{f}$. Since the inclusion $\left(
\ref{finite_oneshot}\right)  $ holds, then any pair of indistinguishable state
evolution of $M$ starting from $\widetilde{S}^{\ast}\cap F_{f}$ is such that
both evolutions in the pair cross the set $\Omega$ within at most $l$ steps.
Therefore at step $k^{\prime}$ we are sure that the actual evolution of $M$ is
such that $x(h)\in\Omega$, for some $h\in\left[  k^{\prime}-(\max\left\{
f,l\right\}  -1),k^{\prime}-\max\left\{  f,l\right\}  )+l\right]  $. Hence
detection occurs with a maximum delay $\delta=\max\left\{  f,l\right\}  -1$,
and with uncertainty $\gamma=l-1$.

The next result characterizes $\Omega$-initial state observability (see
Definition \ref{def_oneshot}), a special case of $\Omega$-diagnosability.

\begin{corollary}
$M$ is $\Omega-$ initial state observable if and only if
\begin{equation}
\left(  X_{0}\cap F^{\ast}\right)  \subset\left(  \Omega\times\Omega\right)
\cup\left(  \overline{\Omega}\times\overline{\Omega}\right)
\label{initial_state}%
\end{equation}

\end{corollary}

\begin{proof}
\textbf{Sufficiency:} \ If condition $\left(  \ref{finite_oneshot}\right)  $
holds with $l=1$, then $\gamma_{1}=\gamma_{2}=0$ and $M$ is $\Omega-$ initial
state observable. Since $\overline{\Lambda_{1}}=\left(  \Omega\times
\Omega\right)  \cup\left(  \overline{\Omega}\times\overline{\Omega}\right)  $
and $\widetilde{S}^{\ast}=X_{0}$, condition $\left(  \ref{finite_oneshot}%
\right)  $ boils down to inclusion \ref{initial_state} and the sufficiency
follows. \textbf{Necessity:} obvious.
\end{proof}

\subsection{Eventual and critical $\Omega-$diagnosability}

Consider the eventual $\Omega-$diagnosability\ property as defined in
Definition \ref{Def_eventual}.

For simplicity, in this sub-section, the sets $B^{\ast}\left(  S^{\ast
}\right)  $ and $B_{k}\left(  S^{\ast}\right)  $ will be denoted by $B^{\ast}$
and $B_{k}$.

A pair $\left(  i,j\right)  $ in the set $\Gamma^{\ast}\cap\Lambda^{\ast}$ is
such that only one state of the pair belongs to $\Omega$. The states \ $i$ and
$j$ are the ending states of a pair of arbitrarily long indistinguishable
state executions of the system $M$, with initial state in $X_{0}$, such that
one of these executions never crosses the set $\Omega$, and are the initial
states of a pair of arbitrarily long indistinguishable state executions of the
same system $M$, such that one of these executions never crosses the set
$\Omega$. Therefore we can prove the following:

\begin{theorem}
\label{Th_main_existence}The FSM $M$ is eventually $\Omega-diag$ if and only
if
\begin{equation}
\Gamma^{\ast}\cap\Lambda^{\ast}=\emptyset\label{diag_ex}%
\end{equation}

\end{theorem}

\begin{proof}
\textbf{Sufficiency:} Let $\tau=g^{\ast}$. By definition of $\Gamma^{\ast}$,
if $\Gamma^{\ast}=\emptyset$ then if for some $k\geq\tau+1$ the execution
$x\in\mathcal{X}_{X_{0}}$ is such that $x(k)\in\Omega$, then any $x^{\prime
}\in\mathbf{y}^{-1}\left(  \mathbf{y}\left(  \left.  x\right\vert _{\left[
1,k\right]  }\right)  \right)  $ is such that $x^{\prime}(h)\in\Omega$, for
some $h\in\left[  k-(g^{\ast}-1),k\right]  $. Therefore condition in
Definition \ref{Def_eventual} is satisfied with parameters $\tau=g^{\ast}$,
$\gamma=g^{\ast}$ and $\delta=g^{\ast}$. By definition of $\Lambda^{\ast}$, if
$\Lambda^{\ast}=\emptyset$ then given a pair $\left(  i,j\right)  \in
\Omega\times\overline{\Omega}$, any pair of indistinguishable state executions
$x^{\prime}$ and $x"$ starting from $\left(  i,j\right)  $, respectively, are
such that $x"(h)\in\Omega$, for some $h\in\left[  1,l^{\ast}\right]  $.
Therefore, $\Lambda^{\ast}=\emptyset$ implies that, if for some $k\geq\tau+1$
the execution $x\in\mathcal{X}_{X_{0}}$ is such that $x(k)\in\Omega$, then any
$x^{\prime}\in\mathbf{y}^{-1}\left(  \mathbf{y}\left(  \left.  x\right\vert
_{\left[  1,k+\max\left\{  f^{\ast},l^{\ast}\right\}  -1\right]  }\right)
\right)  $ is such that $x^{\prime}(h)\in\Omega$, for some $h\in\left[
k,k+(l^{\ast}-1)\right]  $. Therefore condition in Definition
\ref{Def_eventual} is satisfied with parameters $\tau=g^{\ast}$,
$\gamma=l^{\ast}$ and $\delta=l^{\ast}-1$. Finally, it is straightforward to
check that if $\Gamma^{\ast}\cap\Lambda^{\ast}=\emptyset$ then if for some
$k\geq\tau+1$ $x\in\mathcal{X}_{X_{0}}$ is such that $x(k)\in\Omega$, then any
$x^{\prime}\in\mathbf{y}^{-1}\left(  \mathbf{y}\left(  \left.  x\right\vert
_{\left[  1,k+\max\left\{  f^{\ast},l^{\ast}\right\}  -1\right]  }\right)
\right)  $ is such that $x^{\prime}(h)\in\Omega$, for some $h\in\left[
k-(g^{\ast}-1),k+(l^{\ast}-1)\right]  $. Therefore condition in Definition
\ref{Def_eventual} is satisfied with parameters $\tau=g^{\ast}$, $\gamma
=\max\left\{  g^{\ast},l^{\ast}\right\}  -1$ and $\delta=l^{\ast}-1$.

\textbf{Necessity}: Suppose that $\Gamma^{\ast}\cap\Lambda^{\ast}\neq
\emptyset$. Then there exists $\left(  i,j\right)  \in B^{\ast}\cap F^{\ast}$
such that $\left(  i,j\right)  \in\Gamma^{\ast}\cap\Lambda^{\ast}$. Therefore
$\left(  i,j\right)  \in\left(  \left(  \Omega\times\overline{\Omega}\right)
\cup\left(  \Omega\times\overline{\Omega}\right)  \right)  $. Therefore
$\left(  i,j\right)  $ cannot be in general distinguished and there exists two
indistinguishable infinite and left unbounded trajectories crossing $i$ and
$j$, but only one of them crosses the set $\Omega$. Therefore $M$ is not
eventually $\Omega-diag$ and the condition $\left(  \ref{finite_diag}\right)
$ is necessary.
\end{proof}

The following equivalent characterization of eventual $\Omega-$diagnosability
is obtained:

\begin{corollary}
\label{Th_main_finite}If there exist $b\in\left[  1,b^{\ast}\right]  $,
$f\in\left[  1,f^{\ast}\right]  $, $g\in\left[  1,g^{\ast}\right]  $ and
$l\in\left[  1,l^{\ast}\right]  $ such that
\begin{equation}
\left(  B_{b}\cap F_{f}\right)  \subset\overline{\left(  \Gamma_{g}\cap
\Lambda_{l}\right)  } \label{finite_diag}%
\end{equation}
then $M$ is eventually $\Omega-diag$ with $\tau=\max\left\{  b,g\right\}  -1$,
$\delta=\max\left\{  f,l\right\}  -1$, $\gamma_{1}=g-1$, $\gamma_{2}=l-1$.
Conversely, if $M$ is eventually $\Omega-diag$, then there exist $b\in\left[
1,b^{\ast}\right]  $, $f\in\left[  1,f^{\ast}\right]  $, $g\in\left[
1,g^{\ast}\right]  $ and $l\in\left[  1,l^{\ast}\right]  $ such that inclusion
(\ref{finite_diag}) holds.
\end{corollary}

\begin{proof}
Since $\Gamma_{g^{\ast}}=\Gamma^{\ast}\subset B^{\ast}=B_{b^{\ast}}$,
$\Lambda_{l^{\ast}}=\Lambda^{\ast}\subset F^{\ast}=F_{f^{\ast}}$, then
$\Gamma^{\ast}\cap\Lambda^{\ast}=\emptyset$ can be rewritten as $B_{b^{\ast}%
}\cap F_{f^{\ast}}\cap\Gamma_{g^{\ast}}\cap\Lambda_{l^{\ast}}=\emptyset$. This
last condition is equivalent to $\left(  B_{b^{\ast}}\cap F_{f^{\ast}}\right)
\subset\overline{\left(  \Gamma_{g^{\ast}}\cap\Lambda_{l^{\ast}}\right)  }$.
For $b\in\left[  1,b^{\ast}\right]  $, $f\in\left[  1,f^{\ast}\right]  $,
$g\in\left[  1,g^{\ast}\right]  $ and $l\in\left[  1,l^{\ast}\right]  $, if
$\left(  B_{b}\cap F_{f}\right)  \subset\overline{\left(  \Gamma_{g}%
\cap\Lambda_{l}\right)  }$ we can write%
\[
\left(  B_{b^{\ast}}\cap F_{f^{\ast}}\right)  \subset\left(  B_{b}\cap
F_{f}\right)  \subset\overline{\left(  \Gamma_{g}\cap\Lambda_{l}\right)
}\subset\overline{\left(  \Gamma_{g^{\ast}}\cap\Lambda_{l^{\ast}}\right)  }%
\]
Therefore the proof of the sufficiency follows from the proof of Theorem
\ref{Th_main_existence}. The estimation of the parameters $\tau$, $\delta$,
$\gamma_{1}$ and $\gamma_{2}$, given the parameters $b$, $f$, $g$ and $l$ is
obvious. Let us consider the last statement, and suppose that it is false,
i.e. $M$ is eventually $\Omega-diag$ but for any choice of the parameters in
the given interval the inclusion (\ref{finite_diag}) is not satisfied. By
setting $b=b^{\ast}$, $f=f^{\ast}$, $g=g^{\ast}$ and $l=l^{\ast}$ we obtain
that $\Gamma^{\ast}\cap\Lambda^{\ast}\neq\emptyset$ and hence by Theorem
\ref{Th_main_existence} it is not possible that $M$ is eventually
$\Omega-diag$.
\end{proof}

We now show how the condition $\left(  \ref{finite_diag}\right)  $ allows the
determination of the delay of the diagnosis of the crossing event, of the
uncertainty about the time at which the event occurred and of the duration of
the transient where the diagnosis is not possible or not required.

Suppose $\left(  \ref{finite_diag}\right)  $ holds for some $b$, $f$, $g$ and
$l$. Let $k^{\prime}$ be any $k\geq\max\left\{  b,g\right\}  +\max\left\{
f,l\right\}  -1$ such that $\widehat{x}\left(  k\right)  \cap\Omega
\neq\emptyset$. If $\widehat{x}\left(  k^{\prime}\right)  \subset\Omega$, then
we can deduce that the set $\Omega$ was crossed at step $k"=k^{\prime}%
-(\max\left\{  f,l\right\}  -1)$. Otherwise suppose that $\widehat{x}\left(
k^{\prime}\right)  =\left\{  i,j,h\right\}  $, with only the state $i$
belonging to $\Omega$. Then each pair $\left(  i,j\right)  $, $\left(
i,h\right)  $ and $\left(  j,h\right)  $ belongs to $B_{b}\cap F_{f}$. Since
the inclusion $\left(  \ref{finite_diag}\right)  $ holds, then each pair of
state evolutions $x_{1}$ and $x_{2}$, compatible with the observations up to
step $k$, and such that $x_{1}\left(  k"\right)  =i$ and $x_{2}\left(
k")\right)  =j$ has the property that both evolutions crossed the set $\Omega$
in the interval $\left[  k"-g+1,k"\right]  $, if $\left(  i,j\right)
\in\overline{\Gamma_{g}}$, or in the interval $\left[  k",k"+l-1\right]  $, if
$\left(  i,j\right)  \in\overline{\Lambda_{l}}$. Therefore at step $k^{\prime
}$ the actual evolution of $M$ is such that $x(h)\in\Omega$, for some
$h\in\left[  k"-(g-1),k^{\prime}+l-1\right]  $. Hence detection occurs with a
maximum delay $\delta=\max\left\{  f,l\right\}  -1$, with uncertainty
$\gamma=\max\left\{  g,l\right\}  -1$ and with $\tau=\max\left\{  b,g\right\}
-1$. Since $B_{1}=S^{\ast}$, $F_{1}=\Pi$, $\Gamma_{1}=\Lambda_{1}=\left(
\left(  \Omega\times\overline{\Omega}\right)  \cup\left(  \overline{\Omega
}\times\Omega\right)  \right)  $, then if $\left(  \ref{finite_diag}\right)  $
holds in the very special case of $b=g=f=l=1$, then $S^{\ast}\subset\left(
\left(  \Omega\times\Omega\right)  \cup\left(  \overline{\Omega}%
\times\overline{\Omega}\right)  \right)  $, and detection occurs with a
maximum delay $\delta=0$, with uncertainty $\gamma=0$ and with $\tau=0$.

As a consequence of Theorem \ref{Th_main_existence}, we also obtain the
following characterizations of diagnosability in two interesting special
cases. The first one requires no delay in the detection (Case $\delta=0$).

\begin{corollary}
\label{c2}(Case $\delta=0$) $M$ is eventually $\Omega-diag$ with $\delta=0$ if
and only if
\begin{equation}
B^{\ast}\subset\left(  \left(  \Omega\times\Omega\right)  \cup\left(
\overline{\Omega}\times\overline{\Omega}\right)  \right)
\end{equation}

\end{corollary}

\begin{proof}
\textbf{Sufficiency}: if we set $f=1$ and $l=1$, then $\delta=0$. Since
$F_{1}=\Pi$, $B^{\ast}\subset S^{\ast}\subset\Pi$ and $\Gamma_{1}=\Lambda
_{1}=\left(  \left(  \Omega\times\overline{\Omega}\right)  \cup\left(
\overline{\Omega}\times\Omega\right)  \right)  $, then condition
\ref{finite_diag} with $b=b^{\ast}$ and $g=1$ becomes
\[
B^{\ast}\subset\left(  \left(  \Omega\times\Omega\right)  \cup\left(
\overline{\Omega}\times\overline{\Omega}\right)  \right)
\]
and hence $M$ is eventually $\Omega-diag$ with $\delta=0$. \textbf{Necessity}:
suppose that there exists $\left(  i,j\right)  \in B^{\ast}$ such that
$\left(  i,j\right)  \in\overline{\left(  \left(  \Omega\times\Omega\right)
\cup\left(  \overline{\Omega}\times\overline{\Omega}\right)  \right)
}=\left(  \left(  \Omega\times\overline{\Omega}\right)  \cup\left(
\overline{\Omega}\times\Omega\right)  \right)  $. Hence for any $k$ such that
$x(k)=i\in\Omega$, it is not possible to decide at step $k$ if $x(k)\in\Omega$
or not. Hence $M$ is not eventually $\Omega-diag$ with $\delta=0$.
\end{proof}

The second special case requires the exact detection of the step at which the
crossing event occurred (Case $\gamma=0$).

\begin{corollary}
(Case $\gamma=0$) $M$ is eventually $\Omega-diag$ with $\gamma=0$ if and only
if condition $\ref{finite_diag}$ holds with $g=1$ and $l=1$, i.e. there exist
$b$ and $f$ such that%
\begin{equation}
\left(  B_{b}\cap F_{f}\right)  \subset\left(  \Omega\times\Omega\right)
\cup\left(  \overline{\Omega}\times\overline{\Omega}\right)  \label{xx}%
\end{equation}

\end{corollary}

\begin{proof}
Sufficiency is straightforward from Corollary \ref{Th_main_finite}.
\textbf{Necessity}: If $B^{\ast}=\emptyset$ then $M$ is eventually
$\Omega-diag$ with $\gamma=0$ and condition $\left(  \ref{xx}\right)  $ holds
with $b=b^{\ast}$. Otherwise, suppose that for any $b$ and $f$ there exists
$\left(  i,j\right)  \in\left(  B_{b}\cap F_{f}\right)  $, belonging to
$\overline{\left(  \left(  \Omega\times\Omega\right)  \cup\left(
\overline{\Omega}\times\overline{\Omega}\right)  \right)  }=\left(  \left(
\Omega\times\overline{\Omega}\right)  \cup\left(  \overline{\Omega}%
\times\Omega\right)  \right)  $. Therefore there exists $\left(  i,j\right)
\in B^{\ast}\cap F^{\ast}$ such that $\left(  i,j\right)  \in\left(  \left(
\Omega\times\overline{\Omega}\right)  \cup\left(  \overline{\Omega}%
\times\Omega\right)  \right)  $. Hence for any $\tau$ and for any $\delta$
there exists an execution such that whenever $x(k)=i$, $k\geq\tau+1$, it is
not possible to deduce from the output whether $x(k)\in\Omega$. Hence $M$ is
not eventually $\Omega-diag$ with $\gamma=0$.
\end{proof}

Finally, we characterize the property in Definition \ref{Def_critical}. By
Proposition \ref{prop}, a characterization of critical $\Omega-$diagnosability
is obtained as simple consequence of Theorems \ref{Th_LaF_exist} and
\ref{Th_main_existence}:

\begin{corollary}
\label{c1}$M$ is critically $\Omega-diag$ if and only if
\begin{align*}
\widetilde{S}^{\ast}\cap\Lambda^{\ast}  &  =\emptyset\\
&  \text{and}\\
\Gamma^{\ast}\cap\Lambda^{\ast}  &  =\emptyset
\end{align*}

\end{corollary}

\begin{proof}
Straightforward, from Proposition \ref{prop} and equations
(\ref{asymp_oneshot}) and (\ref{diag_ex}).
\end{proof}

The value $\max\left\{  b,g\right\}  -1$ in the statement of Corollary
\ref{Th_main_finite} is not in general the minimum value of $\tau$ such that
$M$ is eventually $\Omega-diag$ (see the next Example \ref{E1}). Hence
critical $\Omega-diag$ cannot be deduced by setting $b=g=1$ in condition
(\ref{finite_diag}).

The next proposition characterizes critical $\Omega-$observability.

\begin{proposition}
(Case $\delta=0$ and $\tau=0$) $M$ is critically $\Omega-obs$ if and only if%
\begin{equation}
S^{\ast}\subset\left(  \Omega\times\Omega\right)  \cup\left(  \overline
{\Omega}\times\overline{\Omega}\right)
\end{equation}

\end{proposition}

\begin{proof}
\textbf{Sufficiency}: Since $B^{\ast}\subset S^{\ast}$ then by Corollary
\ref{c2} $M$ is eventually $\Omega-diag$ with $\delta=0$. Since $F_{f}\subset
S^{\ast}$, then by Corollary \ref{c1} $M$ is eventually $\Omega-diag$ with
$\tau=0$. \textbf{Necessity}: By Corollary \ref{c2}, it is necessary that
$B^{\ast}\subset\left(  \left(  \Omega\times\Omega\right)  \cup\left(
\overline{\Omega}\times\overline{\Omega}\right)  \right)  $. Suppose there
exists $\left(  i,j\right)  \in S^{\ast}$, such that $\left(  i,j\right)
\notin B^{\ast}$ and $\left(  i,j\right)  \in\overline{\left(  \left(
\Omega\times\Omega\right)  \cup\left(  \overline{\Omega}\times\overline
{\Omega}\right)  \right)  }=\left(  \left(  \Omega\times\overline{\Omega
}\right)  \cup\left(  \overline{\Omega}\times\Omega\right)  \right)  $. Then
there exists $k$ such that $x(k)=i\in\Omega$, and it is not possible to decide
at step $k$ if $x(k)\in\Omega$ or not. Hence $M$ is not eventually
$\Omega-diag$ with $\tau=0$.
\end{proof}

\section{Examples\label{examples}}

Recall that for a set $Y\subset X$, $Y^{-}$ denotes the symmetric closure of
$Y$.

\begin{example}
\label{E-1}($M$ is not parametrically $\Omega-diag$) Consider the FSM
represented in Figure \ref{fig_nodiag}. Let $X_{0}=\left\{  1\right\}  $ and
$\Omega=\left\{  3\right\}  $. Since $B^{\ast}\left(  \widetilde{S}^{\ast
}\right)  =\left\{  \left(  3,4\right)  \right\}  ^{-}\cup\Theta$ (see Figure
\ref{tilde_nodiag}) and $\Lambda^{\ast}=\left\{  \left(  3,4\right)  \right\}
^{-}$, then $B^{\ast}\left(  \widetilde{S}^{\ast}\right)  \cap\Lambda^{\ast
}=\left\{  \left(  3,4\right)  \right\}  ^{-}$ and hence by Theorem
\ref{general_as} $M$ is not parametrically $\Omega-diag$.
\end{example}

\begin{example}
\label{E0}($M$ is eventually $\Omega-diag$ but not $\Omega-diag$) Consider the
FSM defined in Example \ref{ex1} and depicted in Figure \ref{fig4}. Let
$X_{0}=X$.
\begin{align*}
\Pi &  =\left\{  \left(  1,3\right)  ,\left(  1,5\right)  ,\left(  3,5\right)
,\left(  2,4\right)  \right\}  ^{-}\cup\Theta\\
S^{\ast}  &  =\Pi\\
B^{\ast}  &  =\left\{  \left(  1,3\right)  \right\}  ^{-}\cup\Theta\text{,
}b^{\ast}=2\\
F^{\ast}  &  =\left\{  \left(  3,5\right)  \right\}  ^{-}\cup\Theta\text{,
}f^{\ast}=2\\
\Gamma^{\ast}  &  =\left\{  \left(  1,3\right)  \right\}  ^{-}\text{, }%
g^{\ast}=2\\
\Lambda^{\ast}  &  =\left\{  \left(  3,5\right)  \right\}  ^{-}\text{,
}l^{\ast}=2
\end{align*}
Since $\Gamma^{\ast}\cap\Lambda^{\ast}=\emptyset$, $M$ is eventually
$\Omega-diag$ with $\delta=1$, $\tau=1$, $\gamma_{1}=\gamma_{2}=1$. Moreover,
$\widetilde{S}^{\ast}=\Pi$, $\widetilde{S}^{\ast}\cap\Lambda^{\ast}%
\neq\emptyset$, and hence $M$ is not $\Omega-diag$. Corollary
\ref{Th_main_finite} allows a better estimation of the parameters. In fact
$B^{\ast}\cap F^{\ast}=\Theta$ which is not a subset of $\Gamma_{1}\cap
\Lambda_{1}$. Therefore $M$ is eventually $\Omega-diag$ with $\delta=1$,
$\tau=1$, $\gamma_{1}=\gamma_{2}=0$.
\end{example}

\begin{example}
\label{E1}($M$ is eventually $\Omega-diag$ or critically $\Omega-diag$,
depending on $X_{0}$) Consider the FSM defined in Example \ref{ex2} and
represented in Figure \ref{fig5}. If $X_{0}=X$ then
\begin{align*}
\Pi &  =\left\{  \left(  1,2\right)  ,\left(  1,3\right)  ,\left(  1,4\right)
,\left(  1,5\right)  ,\left(  2,3\right)  ,\left(  2,4\right)  ,\left(
2,5\right)  ,\left(  3,4\right)  ,\left(  3,5\right)  ,\left(  4,5\right)
,\left(  6,7\right)  \right\}  ^{-}\cup\Theta\\
S^{\ast}  &  =\Pi\\
B^{\ast}  &  =F^{\ast}=\left\{  \left(  2,4\right)  ,\left(  3,5\right)
\right\}  ^{-}\cup\Theta\\
\Gamma^{\ast}  &  =\left\{  \left(  2,4\right)  \right\}  ^{-}\\
\Lambda^{\ast}  &  =\left\{  \left(  3,5\right)  \right\}  ^{-}%
\end{align*}
Therefore $\left(  \Gamma^{\ast}\cap\Lambda^{\ast}\right)  =\emptyset$ and by
Theorem \ref{Th_main_existence} $M$ is eventually $\Omega-diag$. Moreover
$f^{\ast}=b^{\ast}=3$ and $g^{\ast}=l^{\ast}=2$. Hence $M$ is eventually
$\Omega-diag$ with $\delta=\max\left\{  f^{\ast},l^{\ast}\right\}  -1=2$,
$\tau=\max\left\{  b^{\ast},g^{\ast}\right\}  -1=2$, $\gamma_{1}=g^{\ast}%
-1=1$, $\gamma_{2}=l^{\ast}-1=1$.

Since $\widetilde{S}^{\ast}=S^{\ast}=\Pi$ and $\widetilde{S}^{\ast}\cap
F^{\ast}$ is not a subset of $\overline{\Lambda^{\ast}}$, then $M$ is not
$\Omega-diag$, and hence it is not critically $\Omega-diag$.

Suppose $X_{0}=\left\{  1\right\}  $. Then
\begin{align*}
S^{\ast}  &  =\left\{  \left(  2,4\right)  ,\left(  3,5\right)  \right\}
^{-}\cup\Theta\\
B_{1}  &  =B^{\ast}=S^{\ast}\\
F_{1}  &  =F^{\ast}=S^{\ast}\\
\Gamma_{1}  &  =\Lambda_{1}=\left\{  \left(  1,3\right)  ,\left(  1,4\right)
,\left(  2,3\right)  ,\left(  2,4\right)  ,\left(  3,5\right)  ,\left(
4,5\right)  \right\}  ^{-}\\
\Gamma^{\ast}  &  =\left\{  \left(  2,4\right)  \right\}  ^{-}\\
\Lambda^{\ast}  &  =\left\{  \left(  3,5\right)  \right\}  ^{-}\\
\widetilde{S}^{\ast}  &  =\left\{  \left(  2,4\right)  \right\}
\end{align*}
Since $B_{1}\cap F^{\ast}=\left\{  \left(  2,4\right)  ,\left(  3,5\right)
\right\}  $ and $\Gamma_{1}\cap\Lambda^{\ast}=\left\{  \left(  3,5\right)
\right\}  $ then $\left(  B_{1}\cap F^{\ast}\right)  $ is not a subset of
$\overline{\Gamma_{1}\cap\Lambda^{\ast}}$. Therefore there not exist values of
$l$ and $f$ such that the condition
\[
\left(  B_{b}\cap F_{f}\right)  \subset\overline{\Gamma_{g}\cap\Lambda_{l}}%
\]
in Corollary \ref{Th_main_finite} is satisfied for $b=1$ and $g=1$. Hence
$\min\left(  \max\left\{  b,g\right\}  -1\right)  >0$. Nevertheless, $M$ is
critically $\Omega-diag$. In fact $\left(  \Gamma^{\ast}\cap\Lambda^{\ast
}\right)  =\emptyset$ and hence by Theorem \ref{Th_main_existence} $M$ is
eventually $\Omega-diag$. Moreover, $\widetilde{S}^{\ast}\cap\Lambda^{\ast
}=\emptyset$ and hence by Theorem \ref{Th_LaF_exist} $M$ is $\Omega-diag$, and
therefore by Proposition \ref{prop} it is critically $\Omega-diag$, and
therefore it is eventually $\Omega-diag$ with $\tau=0$.
\end{example}

%


\begin{figure}[ptb]
\begin{center}
\includegraphics[scale=0.5]{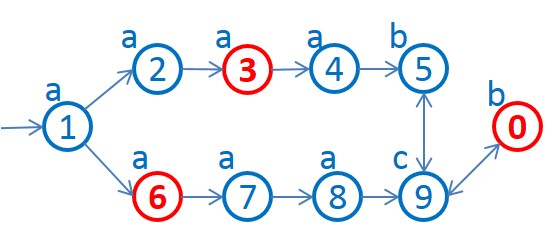}
\end{center}
\caption{FSM $M$ (Example \ref{es_oneshot}).}%
\label{fig_oneshot}%
\end{figure}

\begin{example}
\label{es_oneshot}($M$ is $\Omega-diag$ but not eventually $\Omega-diag$)
Consider the FSM $M$ depicted in Figure \ref{finite_oneshot}. Let
$X_{0}=\left\{  1\right\}  $ and $\Omega=\left\{  3,6,0\right\}  $%
\begin{align*}
S^{\ast}  &  =\left\{  \left(  2,6\right)  ,\left(  3,7\right)  ,\left(
4,8\right)  ,\left(  0,5\right)  \right\}  ^{-}\cup\Theta\\
\widetilde{S}^{\ast}  &  =\left\{  \left(  2,6\right)  \right\}  ^{-}%
\cup\left\{  \left(  1,1\right)  ,\left(  2,2\right)  ,\left(  3,3\right)
,\left(  6,6\right)  \right\} \\
B^{\ast}  &  =\left\{  \left(  0,5\right)  \right\}  ^{-}\\
F_{1}  &  =\Pi\\
F_{2}  &  =\left\{  \left(  2,6\right)  ,\left(  2,3\right)  ,\left(
6,7\right)  ,\left(  0,5\right)  \right\}  ^{-}\cup\Theta\\
F_{3}  &  =\left\{  \left(  2,6\right)  ,\left(  0,5\right)  \right\}  ^{-}\\
F^{\ast}  &  =F_{4}=\left\{  \left(  0,5\right)  \right\}  ^{-}\cup
\Theta\text{, }f^{\ast}=4\\
\Lambda_{1}  &  =\left\{  \left(  2,6\right)  ,\left(  3,7\right)  ,\left(
0,5\right)  \right\}  ^{-}\\
\Lambda^{\ast}  &  =\Lambda_{2}=\left\{  \left(  0,5\right)  \right\}  ^{-}\\
\Gamma^{\ast}  &  =\left\{  \left(  0,5\right)  \right\}  ^{-}%
\end{align*}
This FSM is not eventually $\Omega-diag$, since $\Gamma^{\ast}\cap
\Lambda^{\ast}\neq\emptyset$. It is $\Omega-diag$. In fact by Theorem
\ref{Th_LaF_exist} $\left(  \widetilde{S}^{\ast}\cap F^{\ast}\right)
=\left\{  \left(  1,1\right)  ,\left(  2,2\right)  ,\left(  3,3\right)
,\left(  6,6\right)  \right\}  \subset\overline{\Lambda^{\ast}}$ and hence $M$
is $\Omega-diag$ with $\delta=2$ and $\gamma_{1}=\gamma_{2}=2$. Since
$\widetilde{S}^{\ast}\cap F_{2}\subset\overline{\Lambda_{2}}$, then by
Corollary \ref{Th_LaF_finite}, we can refine our estimations. In fact $M$ is
$\Omega-diag$ with $\delta=1$ and $\gamma_{1}=\gamma_{2}=1$. Moreover $\left(
\widetilde{S}^{\ast}\cap F^{\ast}\right)  \subset\overline{\Lambda_{1}}$, and
hence $M$ is $\Omega-diag$ with $\delta=3$ and $\gamma_{1}=\gamma_{2}=0$.
\end{example}

\section{Appendix: Case $\epsilon\in Y$}

Given the FSM $M$, we propose an algorithm for deriving an FSM $\widehat{M}$
having no silent state such that the parametric diagnosability property can be
checked on either FSMs equivalently. The algorithm in \cite{DeDi:FTSC} can be
retrieved from the one we are going to describe by setting $\Omega=\emptyset$.

Given $M=\left(  X,X_{0},Y,H,\Delta\right)  $, a state $q$ is called silent if
$H\left(  q\right)  =\epsilon$. Let $X_{\epsilon}$ be the set of silent
states. Suppose that any cycle has at least a state $q$ with $H(q)\neq
\epsilon$. Suppose moreover that there is no silent state in $X_{0}$ and that
a state belongs to $X_{0}$ if and only if it has no predecessors. We will say
that a silent state $q\in X$ is reached from $w\in X$ with a silent execution
if there exists a state execution $x\in\mathcal{X}^{\ast}$ such that $x(1)=w$,
$x\left(  \left\vert x\right\vert \right)  =q$ and $\mathbf{y}(x)=H(w)$. The
symbol $\Omega$ denotes the critical set.

The FSM $\widehat{M}=\left(  \widehat{X},\widehat{X}_{0},Y,\widehat{H}%
,\widehat{\Delta}\right)  $ is constructed as described in the following
algorithm (high level description), where, when a state is removed from the
state space, the transitions to and from that state are also be removed,
although this is not explicitly said for the sake of simplicity.

\begin{alg}
\label{alg_nosilent}
\hspace*{\fill} \\

\begin{description}
\item[STEP 0] Split any $q\in X_{\epsilon}$ with silent and nonsilent
successors into two states, $q^{\prime}$ and $q"$, one with only silent
successors, and the other one with only nonsilent successors. The set of
predecessors of $q^{\prime}$ and $q"$ is the same as those of $q$. Update $M$
accordingly. If $q\in\Omega$, then the set $\Omega$ is updated by replacing
$q$ with $q^{\prime}$ and $q"$.

\item[INITIALIZE] $\widehat{M}=M$. Let $X_{F}\subset X\backslash X_{\epsilon}$
be the set of non silent states with a silent successor. Let $X_{L}\subset
X_{\epsilon}$ be the set of silent states with no silent successor.

\item[STEP 1] Split any $q\in X_{L}$ into $2\ast\left\vert X_{F}\right\vert $
states: i.e. in $\widehat{X}$ the set $X_{L}$ is substituted by the set
$\left\{  q_{w},q\in X_{L},w\in X_{F}\right\}  \cup\left\{  \left(
q_{w},1\right)  ,q\in X_{L},w\in X_{F}\right\}  $, where $q_{w}$ is a short
notation for the pair $\left(  q,w\right)  $. Therefore
\[
\widehat{X}=\left(  X\backslash X_{L}\right)  \bigcup\left\{  q_{w},q\in
X_{L},w\in X_{F}\right\}  \bigcup\left\{  \left(  q_{w},1\right)  ,q\in
X_{L},w\in X_{F}\right\}
\]
Moreover $\widehat{H}\left(  q_{w}\right)  =\widehat{H}\left(  \left(
q_{w},1\right)  \right)  =H(w)$, $\forall q\in X_{L}$, $\forall w\in X_{F}$.
The symbol "$1$" appearing in $\left(  q_{w},1\right)  $ is just a flag, whose
meaning will be clarified in the description of STEP 3.

\item[STEP 2] FOR $q\in X_{L}$ DO

\qquad\qquad FOR $w\in X_{F}$ DO

\qquad\qquad\qquad IF $q$ is reached from $w$ with a silent execution of
$M$, and none of the states in this execution belongs to the set $\Omega$,
then $q_{w}\in\widehat{X}$. If $w \in {X}_{0}$ then $q_{w}\in\widehat{X}_{0}$.

\qquad\qquad\qquad OTHERWISE remove $q_{w}$ from the state space
$\widehat{X}$.

\qquad\qquad END

\qquad END

\item[STEP 3] FOR $q\in X_{L}$ DO

\qquad\qquad FOR $w\in X_{F}$ DO

\qquad\qquad\qquad IF $q$ is reached from $w$ with a silent execution of
$M$, and some of the states in this execution belong to the set $\Omega$, then
$\left(  q_{w},1\right)  \in\widehat{X}$. If $w \in {X}_{0}$ then $\left(  q_{w},1\right)\in\widehat{X}_{0}$.

\qquad\qquad\qquad OTHERWISE remove $\left(  q_{w},1\right)  $ from the
state space $\widehat{X}$.

\item[STEP 4] FOR $q_{w}\in\widehat{X}$, $\widehat{\Delta}$ is updated in such
a way that
\[
succ_{\widehat{M}}\left(  q_{w}\right)  =succ\left(  q\right)  \cup\left\{
i_{j}\in\widehat{X}:j\in succ\left(  q\right)  \right\}
\]
and
\[
pre_{\widehat{M}}\left(  q_{w}\right)  =\left(  pre(w)\cap\left(  X\backslash
X_{\epsilon}\right)  \right)  \cup\left\{  i_{j}\in\widehat{X}:i\in pre(w)\cap
X_{\epsilon}\right\}
\]
where, to avoid ambiguities, $succ_{\widehat{M}}$ and $pre_{\widehat{M}}$
denote the operators $succ$ and $pre$ computed for the FSM $\widehat{M}$. If
$w\in succ\left(  q\right)  $ then $\left(  q_{w},q_{w}\right)  \in
\widehat{\Delta}$. A similar construction has to be done for $\left(
q_{w},1\right)  \in\widehat{X}$.

\item[STEP 5] Remove all silent states and all sink states (i.e. states with
no successors) from the state space $\widehat{X}$.

\end{description}
\end{alg}

Define the set $\widehat{\Omega}$ as $\left(  \Omega\cup\left\{  \left(
q_{w},1\right)  ,q\in X_{L},w\in X_{F}\right\}  \right)  \cap\widehat{X}$.

Given $q\in X_{L}$ and $w\in X_{F}$, the test in STEP 2 ($q$ is reached from
$w$ with a silent execution, and none of the states in this execution belongs
to the set $\Omega$) can be done with the procedure described in $\left(
\ref{Pro1}\right)  $, where $\lambda$ is the maximal length of a silent string
(recall that there are no silent cycles in $M$)%
\begin{equation}%
\begin{array}
[c]{l}%
C(0)=\left\{  q\right\}  \text{, }k=0\\
\text{WHILE }\left(  k<\lambda\right)  \wedge\left(  w\notin C\left(
k\right)  \right)  \\
\text{DO }k=k+1;C\left(  k\right)  =\bigcup\limits_{z\in\left(  C\left(
k-1\right)  \right)  \cap\left(  X_{\epsilon}\backslash\Omega\right)
}pre\left\{  z\right\}  \\
\text{END}%
\end{array}
\label{Pro1}%
\end{equation}

\begin{proposition}
\label{P1}Given $M$, a state $q\in X_{\epsilon}\backslash\Omega$ is reached
from $w\in X\backslash\left(  X_{\epsilon}\cup\Omega\right)  $ with a silent
execution, and none of the states in this execution belongs to the set
$\Omega$, if and only if the exit condition of the cycle defined in $\left(
\ref{Pro1}\right)  $ is $w\in C\left(  \overline{k}\right)  $, $1\leq
\overline{k}\leq\lambda$.
\end{proposition}

The proof of the above Proposition is obvious and hence can be omitted.

We now describe how the test in STEP 3 ($q$ is reached from $w$ with a silent
execution, and some of the states in this execution belong to the set $\Omega
$) can be performed. The procedure described in $\left(  \ref{Pro2}\right)  $
is instrumental in computing the set $\left(  \bigcup\limits_{k=1...g}V\left(
k\right)  \right)  $, for a given $g\in\left[  2,\lambda\right]  $, required
in the statement of the next Proposition \ref{P2}%

\begin{equation}%
\begin{array}
[c]{l}%
G(1)=\left\{  q\right\}  \text{, }k=1\\
\text{FOR }k=1..g-1\text{ DO}\\
G\left(  k+1\right)  =\bigcup\limits_{z\in G\left(  k\right)  \cap
X_{\epsilon}}pre\left\{  z\right\}  \text{, }k=k+1\\
\text{END}\\
\text{IF }w\in G(g)\text{ THEN\ }\\
V(1)=G\left(  g\right)  \\
\text{FOR }k=1...g-1\text{ DO}\\
V\left(  k+1\right)  =\left(  \bigcup\limits_{z\in V\left(  k\right)
}succ\left\{  z\right\}  \right)  \cap\left(  G\left(  g+k\right)  \right)
\text{ END}\\
\text{OTHERWISE }\\
\text{FOR }h=1...g\text{ DO }V\left(  k\right)  =\emptyset
\end{array}
\label{Pro2}%
\end{equation}

\begin{proposition}
\label{P2}Given $M$, a state $q\in X_{\epsilon}$ is reached from $w\in
X\backslash X_{\epsilon}$ with a silent execution, and some of the states in
this execution belong to the set $\Omega$ if and only if there exists
$g\in\left[  2,\lambda\right]  :\left(  \bigcup\limits_{k=1...g}V\left(
k\right)  \right)  \bigcap\Omega\neq\emptyset$.
\end{proposition}

\begin{proof}
Given the set $\mathcal{S\subset X}^{\ast}$ of all the finite silent
executions $x$ of $M$, of length $g$, with first state equal to $w$ and last
state equal to $q$, the recursion in equation $\left(  \ref{Pro2}\right)  $
defines the sets $V\left(  k\right)  $, $k=1...g$. By construction, $V\left(
k\right)  =\left\{  q\in X:x(k)=q\wedge x\in\mathcal{S}\right\}  $. Therefore
the result follows.
\end{proof}

On the basis of the above Propositions \ref{P1} and \ref{P2} and of the fact
that after the execution of STEP 0 of Algorithm  \ref{alg_nosilent} the number
of states in $X$ is less than $2N$, where $N$ is the original number of states
of $X$, the following proposition holds:

\begin{proposition}
The complexity of the Algorithm \ref{alg_nosilent} is $O\left(  N^{3}\right)
$.
\end{proposition}

\begin{proof}
Straightforward.
\end{proof}

We now establish a precise relationship between the FSM $M$ and the FSM
$\widehat{M}$. Recall that $P\left(  \sigma\right)  $ is the projection of the
string $\sigma$, i.e. the string obtained from $\sigma$ by erasing the symbol
$\epsilon$.

By construction,

- for any finite state execution $x$ of $M$ there exists a finite state
execution $\widehat{x}$ of $\widehat{M}$, such that $\widehat{x}=P\left(
x\right)  $

- for any finite state execution $x$ of $M$ such that $x(k)\notin\Omega$,
$\forall k\in\left[  1,\left\vert x\right\vert \right]  $, $\widehat{x}%
(k)\notin\widehat{\Omega}$, $\forall k\in\left[  1,\left\vert \widehat{x}%
\right\vert \right]  $, $\widehat{x}=P\left(  x\right)  $

- for any finite state execution $x$ of $M$ for which there exists $\delta
\in\left[  1,\left\vert x\right\vert \right]  $ such that $x(\left\vert
x\right\vert -\delta)\in\Omega$, $\exists\widehat{\delta}\leq\delta
:\widehat{x}(\left\vert \widehat{x}\right\vert -\widehat{\delta}%
)\in\widehat{\Omega}$, $\widehat{x}=P\left(  x\right)  $

Conversely,

- for any finite state execution $\widehat{x}$ of $\widehat{M}$ there exists a
set of finite state executions $x$ of $M$, such that $P\left(  x\right)
=\widehat{x}$

- for any finite state execution $\widehat{x}$ of $\widehat{M}$ with
$\widehat{x}(k)\notin\widehat{\Omega}$, $\forall k\in\left[  1,\left\vert
\widehat{x}\right\vert \right]  $, $x(k)\notin\Omega$, $\forall k\in\left[
1,\left\vert x\right\vert \right]  $, $\forall x:P\left(  x\right)
=\widehat{x}$

- given any finite state execution $\widehat{x}$ of $\widehat{M}$ for which
there exists $\delta\in\left[  1,\left\vert x\right\vert \right]  $ such that
$\widehat{x}(\left\vert \widehat{x}\right\vert -\delta)\in\widehat{\Omega}$,
for any finite state exection $x$ of $M$, with $P\left(  x\right)
=\widehat{x}$, there exists $\delta_{x}\geq\delta:x(\left\vert x\right\vert
-\delta_{x})\in\Omega$

Finally,

- $\mathbf{y}\left(  \widehat{x}\right)  =\mathbf{y}\left(  x\right)  $,
$\forall\widehat{x}=P\left(  x\right)  $

- $\mathbf{y}\left(  x\right)  =\mathbf{y}\left(  \widehat{x}\right)  $,
$\forall x:P\left(  x\right)  =\widehat{x}$

Therefore we can establish the following result:

\begin{proposition}
\label{Prop_nosilent}The FSMs $M$ and $\widehat{M}$ have the same output
language. If $M$ is parametrically $\Omega-diag$ with parameters $\tau$,
$\delta$, $\gamma$ and $T\in\left\{  0,\infty\right\}  $, then there exist
$\widehat{\tau}\leq\tau$, $\widehat{\delta}\leq\delta$, $\widehat{\gamma}%
\leq\gamma$ such that $\widehat{M}$ is parametrically $\widehat{\Omega}-diag$
with parameters $\widehat{\tau}$, $\widehat{\delta}$, $\widehat{\gamma}$ and
$\widehat{T}=T$. Conversely, if $\widehat{M}$ is parametrically
$\widehat{\Omega}-diag$ with parameters $\widehat{\tau}$, $\widehat{\delta}$,
$\widehat{\gamma}$ and $\widehat{T}\in\left\{  0,\infty\right\}  $, then there
exist $\tau\geq\widehat{\tau}$, $\delta\geq\widehat{\delta}$ and $\gamma
\geq\widehat{\gamma}$ such that $M$ is parametrically $\Omega-diag$ with
parameters $\tau$, $\delta$, $\gamma$ and $T=\widehat{T}$.
\end{proposition}

\begin{example}
Consider the FSM $M$ in Fig. \ref{figura8}, where the state $3$ is critical
(i.e. $\Omega=\left\{  3\right\}  $) and silent (i.e. $X_{\epsilon}=\left\{
3\right\}  $). Let $X_{0}=\left\{  0,4\right\}  $. By direct inspection $M$ is
eventually $\Omega-diag$, with $\tau=2$ and $\delta=1$.%

\begin{figure}[ptb]
\begin{center}
\includegraphics[scale=0.5]{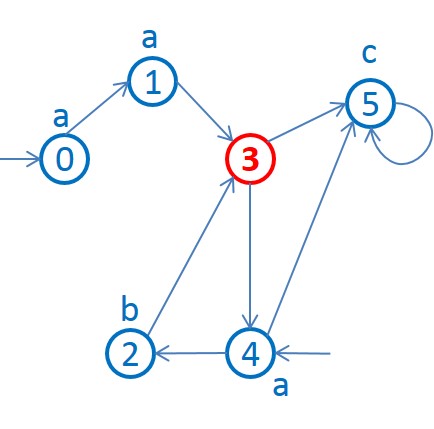}
\end{center}
\caption{FSM $M$.}%
\label{figura8}%
\end{figure}

By applying Algorithm \ref{alg_nosilent}, $X_{F}=\left\{  1,2\right\}  $ and
$X_{L}=\left\{  3\right\}  $. The state $3$ is removed from the state space
and substituted with the states called $3_{1}$, $3_{2}$, $\left(
3_{1},1\right)  $ and $\left(  3_{2},1\right)  $. Since there is no silent
execution starting from $w\in X_{F}$ and reaching the state $3$ without
crossing the set $\Omega$, then only $\left(  3_{1},1\right)  $ and $\left(
3_{2},1\right)  $ have to be considered. In Fig. \ref{figura9} such states are
renamed $6$ and $7$, respectively, for simplicity

\begin{figure}[ptb]
\begin{center}
\includegraphics[scale=0.5]{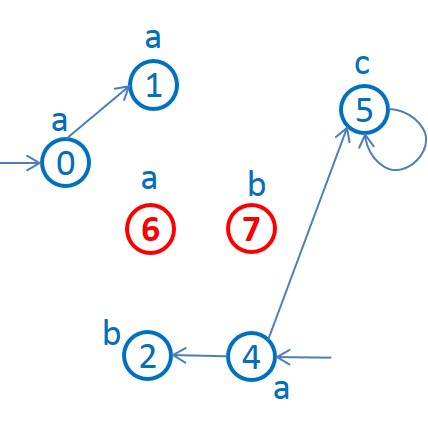}
\end{center}
\caption{FSM after state space redefinition (STEP 1 of Algorithm \ref{alg_nosilent}).}%
\label{figura9}%
\end{figure}

After STEP 4, we obtain the FSM in Fig. \ref{figura10}.%

\begin{figure}[ptb]
\begin{center}
\includegraphics[scale=0.5]{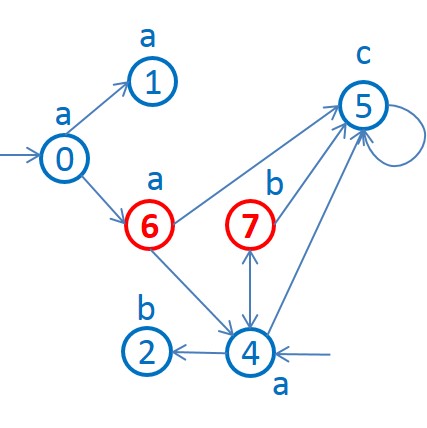}
\end{center}
\caption{FSM after STEP 4 (Algorithm \ref{alg_nosilent}).}%
\label{figura10}%
\end{figure}

Finally, we can remove the sink states $1$ and $2$, and the resulting FSM
$\widehat{M}$ is depicted in Fig. \ref{figura11}, and $\widehat{\Omega
}=\left\{  6,7\right\}  $. Then $\widehat{M}$ is eventually $\widehat{\Omega
}-diag$, with $\widehat{\tau}=1\leq\tau$ and $\widehat{\delta}=1\leq\delta$,
in accord to Proposition \ref{Prop_nosilent}%
\begin{figure}[ptb]
\begin{center}
\includegraphics[scale=0.5]{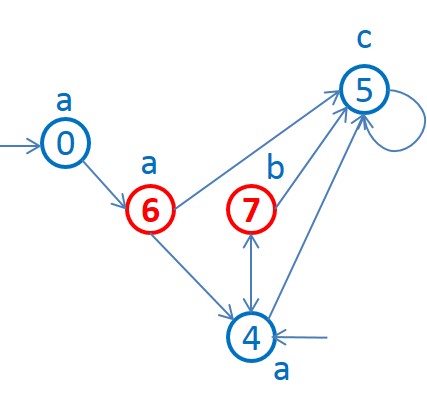}
\end{center}
\caption{FSM $\protect\widehat{M}$.}%
\label{figura11}%
\end{figure}

In the following table we see how state trajectories of $M$ are mapped on
state trajectories of $\widehat{M}$.
\[%
\begin{tabular}
[c]{|l|l|l|}\hline
$x$ & $\widehat{x}$ & $\mathbf{y}\left(  x\right)  =\mathbf{y}\left(
\widehat{x}\right)  $\\\hline
$0135^{\ast}$ & $065^{\ast}$ & $aac^{\ast}$\\\hline
$013\left(  423\right)  ^{\ast}5^{\ast}$ & $06\left(  47\right)  ^{\ast
}5^{\ast}$ & $aa\left(  ab\right)  ^{\ast}c^{\ast}$\\\hline
$01345^{\ast}$ & $0645^{\ast}$ & $ac^{\ast}$\\\hline
$4\left(  234\right)  ^{\ast}5^{\ast}$ & $4\left(  74\right)  ^{\ast}5^{\ast}$
& $a\left(  ab\right)  ^{\ast}c^{\ast}$\\\hline
$4235^{\ast}$ & $475^{\ast}$ & $abc^{\ast}$\\\hline
$45^{\ast}$ & $45^{\ast}$ & $ac^{\ast}$\\\hline
\end{tabular}
\ \
\]

\end{example}

\section{Conclusions}

In this paper, we proposed a general framework for the analysis and
characterization of observability and diagnosability of finite state systems.
Observability and diagnosability were defined with respect to a subset of the
state space, called critical set, i.e. a set of discrete states representing a
set of faults, or more generally any set of interest. Using the proposed
framework, it is possible to check diagnosability of a critical event and at
the same time compute the delay of the diagnosis with respect to the
occurrence of the event, the uncertainty about the time at which that event
occurred, and the duration of a possible initial transient where the diagnosis
is not possible or nor required. Moreover, in the unifying framework we
propose, it was possible to precisely compare some of observabllity and
diagnosability notions existing in the literature and the ones we introduced
in our paper.

For discrete event systems, some effort in the direction of a decentralized
approach to observability and diagnosability has been made e.g. in
\cite{Cieslak:1988} and in \cite{Rudie:1989} where the observability with
respect to a language \cite{Lin:1988} was generalized to the case of
decentralized systems by introducing the notion of coobservability. In
\cite{Wang:2011} and in \cite{Yin:2015} it was proven that coobservability and
codiagnosability can be mapped from one to the other. Extending our approach
to a decentralized framework and comparing our results with the ones above
will be the subject of future investigation.

\bibliographystyle{plain}
\bibliography{biblioElena}

\end{document}